\RequirePackage{amsthm}% must precede sn-jnl.cls so that its theorem styles (thmstyleone, ...) are defined
\documentclass[pdflatex,sn-mathphys-num]{sn-jnl}% Math and Physical Sciences Numbered Reference Style

\usepackage{amsmath,amssymb,amsthm}
\usepackage{mathtools}
\usepackage{booktabs}
\usepackage{tikz}
\hypersetup{pdftitle={A refinement of the edge theorem for order and chain polytopes},pdfauthor={Jon Lee}}

\theoremstyle{thmstyleone}
\newtheorem{theorem}{Theorem}
\newtheorem{lemma}[theorem]{Lemma}
\newtheorem{corollary}[theorem]{Corollary}
\newtheorem{proposition}[theorem]{Proposition}
\newtheorem{conjecture}[theorem]{Conjecture}
\theoremstyle{thmstylethree}
\newtheorem{case}{Case}
\newtheorem{example}[theorem]{Example}
\theoremstyle{thmstyletwo}
\newtheorem{remark}[theorem]{Remark}

\newcommand{\Ocal}{\mathcal{O}}
\newcommand{\Ccal}{\mathcal{C}}
\newcommand{\RR}{\mathbb{R}}
\newcommand{\down}{{\downarrow}}
\newcommand{\up}{{\uparrow}}
\DeclareMathOperator{\ind}{ind}
\newcommand{\minc}{\min\nolimits^{\circ}}
\newcommand{\cl}[1]{\overline{#1}}
\begin{document}

\title[Refining the edge theorem for order and chain polytopes]{A refinement of the edge theorem\break for order and chain polytopes}

\author*[1]{\fnm{Jon} \sur{Lee}}\email{jonxlee@umich.edu}
% ORCID 0000-0002-8190-1091: enter it in Editorial Manager at submission (Springer links it to the published article)

\affil*[1]{\orgname{University of Michigan}, \orgaddress{\city{Ann Arbor}, \state{MI}, \country{USA}}}

%% Order asks for an abstract of 150--250 words. The original (Elsevier) abstract is kept below, commented out.
\abstract{For a finite poset, we partition the edges of the order polytope and of the chain polytope into classes indexed by the connected convex subsets of the poset, and we establish that corresponding classes have the same cardinality. This yields a closed formula for the common number of edges, and it identifies the bijection of Hibi, Li, Sahara and Shikama as a disjoint union of simple bijections between corresponding classes, explaining why it is a bijection and how it acts on edge directions and lengths. As consequences, the two polytopes have equally many edge directions, with matching multiplicities, and, for every $m$, equally many pairs of vertices whose symmetric difference has $m$ components; the chain polytope has at least as many edges parallel to each coordinate subspace as the order polytope; and the edge lengths of the chain polytope are dominated by those of the order polytope, with strict inequality in the last two comparisons exactly when the poset contains a three-element chain. We discuss implications for linear and convex combinatorial optimization over ideals and antichains. We also express the number of edges in terms of the comparability graph, extending to stable-set polytopes of all graphs; give a recursion for series-parallel posets and closed formulas for layered, zigzag and crown posets; and characterize the distributive lattices for which the two polytopes are unimodularly equivalent. Finally, either diameter can exceed the other arbitrarily; both are bounded by the width, and they coincide for series-parallel, zigzag and crown posets.}



\maketitle

%=====================================================================
\section{Introduction}
%=====================================================================

Throughout, $P=\{a_1\,,\dots,a_d\}$ is a finite poset with partial order $\preceq$; we write $a\prec b$ if $a\preceq b$ and $a\neq b$, and we reserve $\le$ and $<$ for numbers. For $W\subseteq P$, we let $\rho(W):=\sum_{i\,:\,a_i\in W}\mathbf e_i\in\RR^d$, where $\mathbf e_1\,,\dots,\mathbf e_d$ are the standard unit vectors. A subset $I\subseteq P$ is an \emph{order ideal} (or down-set; we say simply \emph{ideal}) if $a\in I$ and $b\preceq a$ imply $b\in I$, and a subset $U\subseteq P$ is an \emph{up-set} if $a\in U$ and $b\succeq a$ imply $b\in U$; a subset $A\subseteq P$ is an \emph{antichain} if its elements are pairwise incomparable. The empty set is both an ideal and an antichain. Following Hibi, Li, Sahara and Shikama~\cite{HLSS}, the \emph{order polytope} and the \emph{chain polytope} of $P$ are
\begin{align*}
  \Ocal(P)&:=\{x\in\RR^d:\ 0\le x_i\le 1 \text{ for all } i,\ x_i\ge x_j \text{ whenever } a_i\preceq a_j\},\\
  \Ccal(P)&:=\Bigl\{x\in\RR^d:\ x\ge 0,\ \textstyle\sum_{i\,:\,a_i\in K}x_i\le 1 \text{ for every chain } K \text{ of } P\Bigr\}.
\end{align*}
(\citet{Stanley} uses the reverse inequalities in $\Ocal(P)$; the two versions are exchanged by $x\mapsto\mathbf 1-x$. The order polytope appears earlier in \citet{Geissinger}.) For a finite poset $Q$, we let $\mathcal J(Q)$ denote the set of ideals of $Q$, which is a distributive lattice under inclusion, so that $|\mathcal J(Q)|$ is the number of ideals of $Q$. Because $I\mapsto\max I$ is a bijection from the ideals of $Q$ to the antichains of $Q$, $|\mathcal J(Q)|$ is also the number of antichains of $Q$; note that $|\mathcal J(\emptyset)|=1$.

For $W\subseteq P$, we let $\down W:=\{a: a\preceq b,\ \exists\, b\in W\}$ and $\up W:=\{a: a\succeq b,\ \exists\, b\in W\}$, and we let $\max W$ and $\min W$ denote the sets of maximal and minimal elements of $W$ in the induced order. For $a\preceq b$, the \emph{interval} is $P[a,b]:=\{c: a\preceq c\preceq b\}$; we say that $b$ \emph{covers} $a$ if $a\prec b$ and $P[a,b]=\{a,b\}$, and we describe small posets by their cover relations, as drawn in their Hasse diagrams. The set $W$ is \emph{convex} if $a\preceq c\preceq b$ with $a,b\in W$ implies $c\in W$, and its \emph{convex closure} is $\cl W:=\{c: b\preceq c\preceq b',\ \exists\, b,b'\in W\}$, the smallest convex set containing $W$. The set $W$ is \emph{connected} if any two elements of $W$ are joined by a sequence of elements of $W$ in which consecutive elements are comparable; the maximal connected subsets of $P$ are its \emph{connected components}. A \emph{three-element chain} in $W$ is a triple $a\prec b\prec c$ of elements of $W$. The \emph{height} (resp., \emph{width}) of $P$ is the largest cardinality of a chain (resp., antichain) of $P$; we write $w(P)$ for the width. Finally, for $S\subseteq P$ we define
\[
  S^{\parallel}:=\{c\in P: c \text{ is incomparable to every element of } S\}=P\setminus(\up S\cup\down S),
\]
which we regard as an induced subposet of $P$. Note that $S\cap S^{\parallel}=\emptyset$.

The \emph{comparability graph} $G(P)$ of $P$ has vertex set $P$, with distinct elements adjacent when they are comparable. A \emph{stable set} (or independent set) of a graph is a set of pairwise nonadjacent vertices, and the \emph{stable-set polytope} of a graph is the convex hull of the characteristic vectors of its stable sets. The antichains of $P$ are exactly the stable sets of $G(P)$, so $\Ccal(P)$ is the stable-set polytope of $G(P)$ (see \cite{Stanley}).

For a polytope $\mathcal Q$, we write $f_k(\mathcal Q)$ for the number of $k$-dimensional faces of $\mathcal Q$, and $(f_0(\mathcal Q),f_1(\mathcal Q),\dots)$ for its \emph{$f$-vector}. Two full-dimensional lattice polytopes $\mathcal Q,\mathcal Q'\subseteq\RR^d$ are \emph{unimodularly equivalent} if $\mathcal Q'=U\mathcal Q+v$ for some integer matrix $U$ with $\det U=\pm1$ and some $v\in\mathbb Z^d$.

Let $X$ denote the five-element poset with $a,b\prec c\prec g,h$ (Figure~\ref{fig:X}). Following \citet{HL-unimod}, we say that $X$ \emph{appears in} $P$ if $P$ has elements $a,b,c,g,h$ with $a,b\prec c\prec g,h$, $a$ incomparable to $b$, and $g$ incomparable to $h$; otherwise $P$ is \emph{$X$-free}.

\begin{figure}[ht]
\centering
\begin{tikzpicture}[scale=0.75,every node/.style={circle,draw,inner sep=1.4pt,fill=white}]
  \node[label=left:$a$] (a) at (-1,0) {};
  \node[label=right:$b$] (b) at (1,0) {};
  \node[label=right:$c$] (c) at (0,1) {};
  \node[label=left:$g$] (g) at (-1,2) {};
  \node[label=right:$h$] (h) at (1,2) {};
  \draw (a)--(c)--(b); \draw (g)--(c)--(h);
\end{tikzpicture}
\caption{The Hasse diagram of the poset $X$.}\label{fig:X}
\end{figure}

We first recall relevant results on vertices and face numbers.

\begin{theorem}[\cite{Stanley}]\label{thm:stanley}
The vertices of $\Ocal(P)$ are the points $\rho(I)$ for the ideals $I$ of $P$, and the vertices of $\Ccal(P)$ are the points $\rho(A)$ for the antichains $A$ of $P$.
\end{theorem}

\begin{theorem}[\cite{HL-unimod}]\label{thm:X}
The following are equivalent: $\Ocal(P)$ and $\Ccal(P)$ are unimodularly equivalent; $\Ocal(P)$ and $\Ccal(P)$ have the same $f$-vector; $P$ is $X$-free.
\end{theorem}

For distributive lattices, the condition of Theorem~\ref{thm:X} can be read off from a much smaller poset. By the fundamental theorem for finite distributive lattices, every finite distributive lattice $L$ is isomorphic to $\mathcal J(Q)$ for a finite poset $Q$, unique up to isomorphism, namely the subposet of join-irreducible elements of $L$ (see \cite[Theorem~3.4.1]{EC1}). The lattice $\mathcal J(Q)$ can have up to $2^{|Q|}$ elements, so it is useful to be able to decide from $Q$ alone whether $X$ appears in $\mathcal J(Q)$.

\begin{proposition}\label{prop:distributive}
Let $Q$ be a finite poset. Then $X$ appears in $\mathcal J(Q)$ if and only if $Q$ has an ideal $I$ such that neither $I$ nor $Q\setminus I$ is a chain. Consequently, $\Ocal(\mathcal J(Q))$ and $\Ccal(\mathcal J(Q))$ are unimodularly equivalent if and only if, for every ideal $I$ of $Q$, at least one of $I$ and $Q\setminus I$ is a chain.
\end{proposition}

\begin{proof}
$X$ appears in a poset if and only if some element has two incomparable elements strictly below it and two incomparable elements strictly above it. In $\mathcal J(Q)$, the elements strictly below $I$ are the proper ideals of $I$, and these form a chain if and only if $I$ is a chain; the elements strictly above $I$ are the sets $I\cup K$ with $K$ a nonempty ideal of $Q\setminus I$, and these form a chain if and only if $Q\setminus I$ is a chain. The second statement follows from Theorem~\ref{thm:X}.
\end{proof}

For example, if $A_n$ denotes an $n$-element antichain, then the Boolean lattice $\mathcal J(A_n)$ is $X$-free if and only if $n\le3$.

We now turn to edges. We write $[u,v]$ for the line segment joining points $u,v\in\RR^d$, and $\|v\|$ for the Euclidean norm. An \emph{edge} of a polytope is a face of dimension~$1$; it is a segment $[u,v]$ joining two vertices $u\neq v$, which are then called \emph{adjacent}. Because every vertex of $\Ocal(P)$ or $\Ccal(P)$ is of the form $\rho(W)$, every edge of these polytopes has the form $[\rho(W),\rho(W')]$, and determining the edges amounts to deciding which pairs of vertices are adjacent. The vertices and edges of a polytope form its \emph{graph}; the \emph{distance} between two vertices is the least number of edges in a path joining them in this graph, and the \emph{diameter} of the polytope is the largest such distance. A \emph{direction} of an edge $[u,v]$ is any nonzero multiple of $v-u$, and two edges are \emph{parallel} if they have a common direction. By the \emph{number of edge directions} of a polytope we mean its number of parallel classes of edges, and the \emph{multiplicity} of an edge direction is the size of the corresponding class.

\begin{theorem}[{\cite[Theorem~1.1]{IT}}; see also \cite{HL-cutting}]\label{thm:HL}
Let $I\neq J$ be ideals of $P$, and let $A\neq B$ be antichains of $P$. Then $[\rho(I),\rho(J)]$ is an edge of $\Ocal(P)$ if and only if, after possibly interchanging $I$ and $J$,
\[
  (\mathrm E_{\Ocal})\qquad I\subseteq J \ \text{ and } \ J\setminus I \text{ is connected};
\]
and $[\rho(A),\rho(B)]$ is an edge of $\Ccal(P)$ if and only if
\[
  (\mathrm E_{\Ccal})\qquad A\triangle B \text{ is connected}.
\]
\end{theorem}

The first part of Theorem~\ref{thm:HL} was derived by \citet{IT} from the colouring criterion for adjacency on $0/1$-polytopes of \citet{HK} (see also \cite{Hausmann}), and was found again by \citet{HL-cutting}. The second part is the special case $G=G(P)$ of the theorem of \citet{Chvatal} that two stable sets $A,B$ of a graph $G$ give adjacent vertices of its stable-set polytope if and only if $A\triangle B$ induces a connected subgraph of $G$; this theorem also follows from the colouring criterion (see \cite{Hausmann,IT}).

The graphs of $\Ocal(P)$ and $\Ccal(P)$ have the same number of edges (Theorem~\ref{thm:HLSS} below), but not necessarily the same diameter. If $P$ is $X$-free, the graphs are isomorphic by Theorem~\ref{thm:X}, but when $X$ appears in $P$, either graph can have the larger diameter; we return to diameters in Section~\ref{sec:diam}.

\begin{example}\label{ex:diam}
If $P$ consists of $X$ together with elements $p,q$ with covers $a\prec p\prec g$ and $a\prec q\prec h$, then $\operatorname{diam}\Ocal(P)=2<3=\operatorname{diam}\Ccal(P)$; indeed, no element is comparable to each of $b,p,q$, so $\emptyset$ and $\{b,p,q\}$ are at distance $3$ in the graph of $\Ccal(P)$. If $P$ consists of $X$ together with elements $p,q,r,s$ with covers $a\prec p\prec g$, $q\prec h$, $q\prec r$, $b\prec r$ and $s\prec r$, then $\operatorname{diam}\Ocal(P)=4>3=\operatorname{diam}\Ccal(P)$. We verify these values in Proposition~\ref{prop:ex5}. Moreover, the difference can be arbitrarily large: because $\Ocal(P+Q)=\Ocal(P)\times\Ocal(Q)$ and $\Ccal(P+Q)=\Ccal(P)\times\Ccal(Q)$ for the disjoint union $P+Q$, and the graph of a product of polytopes is the Cartesian product of their graphs, both diameters are additive over disjoint unions, so the disjoint union of $k$ copies of the first poset above has $\operatorname{diam}\Ocal=2k$ and $\operatorname{diam}\Ccal=3k$, and that of $k$ copies of the second has $\operatorname{diam}\Ocal=4k$ and $\operatorname{diam}\Ccal=3k$.
\end{example}

\begin{theorem}[\cite{HLSS}]\label{thm:HLSS}
For every finite poset $P$, $\Ocal(P)$ and $\Ccal(P)$ have the same number of edges.
\end{theorem}

The proof of Theorem~\ref{thm:HLSS} in \cite{HLSS} is given by a bijection between the pairs satisfying $(\mathrm E_{\Ocal})$ and the pairs satisfying $(\mathrm E_{\Ccal})$. Our purpose is to demonstrate that this bijection decomposes along a natural partition of the edges, and to draw consequences of the decomposition.

For an edge $[\rho(I),\rho(J)]$ of $\Ocal(P)$ with $I\subseteq J$, we call $J\setminus I$ its \emph{difference set}; for an edge $[\rho(A),\rho(B)]$ of $\Ccal(P)$, we call $A\triangle B$ its \emph{difference set}. For $S,T\subseteq P$, let $E_{\Ocal}(S)$ be the set of edges $[\rho(I),\rho(J)]$ of $\Ocal(P)$ with $I\subseteq J$ and $J\setminus I=S$, and let $E_{\Ccal}(T)$ be the set of edges $[\rho(A),\rho(B)]$ of $\Ccal(P)$ with $A\triangle B=T$. Further, let
\begin{align*}
  \mathfrak S(P)&:=\{S\subseteq P: S\neq\emptyset,\ S \text{ connected and convex}\},\\
  \mathfrak T(P)&:=\{T\subseteq P: T\neq\emptyset,\ T \text{ connected and containing no three-element chain}\}.
\end{align*}

Our main result is the following refinement of Theorem~\ref{thm:HLSS}.

\begin{theorem}\label{thm:main}
Let $P$ be a finite poset.
\begin{enumerate}[(iii)]
\renewcommand{\theenumi}{\roman{enumi}}\renewcommand{\labelenumi}{\textup{(\theenumi)}}
\item $E_{\Ocal}(S)\neq\emptyset$ if and only if $S\in\mathfrak S(P)$, in which case $|E_{\Ocal}(S)|=|\mathcal J(S^{\parallel})|$.
\item $E_{\Ccal}(T)\neq\emptyset$ if and only if $T\in\mathfrak T(P)$, in which case $|E_{\Ccal}(T)|=|\mathcal J(T^{\parallel})|$.
\item The map $\Phi: S\mapsto \max S\cup\min S$ is a bijection from $\mathfrak S(P)$ to $\mathfrak T(P)$, with inverse $T\mapsto\cl{T}$, and $\Phi(S)^{\parallel}=S^{\parallel}$.
\end{enumerate}
\end{theorem}

Theorem~\ref{thm:main} yields Theorem~\ref{thm:HLSS}: every edge has exactly one difference set, so by (i) and (ii) the numbers of edges of $\Ocal(P)$ and $\Ccal(P)$ are $\sum_{S\in\mathfrak S(P)}|\mathcal J(S^{\parallel})|$ and $\sum_{T\in\mathfrak T(P)}|\mathcal J(T^{\parallel})|$, and by (iii) the bijection $\Phi$ matches these two sums term by term. This gives a formula for the common number of edges:
\begin{equation}\label{eq:formula}
  f_1(\Ocal(P)) \;=\; \sum_{S\in\mathfrak S(P)} |\mathcal J(S^{\parallel})| \;=\; f_1(\Ccal(P)).
\end{equation}

Theorem~\ref{thm:main} is the edge analogue of a parametrization of triangular $2$-faces obtained recently in \cite{FLM}. There, the triangles of $\Ocal(P)$ and of $\Ccal(P)$ are indexed by a connected convex subposet $Q$ of height at least two, an antichain $W$ of elements incomparable to $Q$, and a local object inside $Q$, so that each count is a sum over $Q$ of $|\mathcal J(Q^{\parallel})|$ times a local invariant of $Q$; the local invariant for $\Ocal(P)$ is at most that for $\Ccal(P)$, with strict inequality for some $Q$ exactly when $X$ appears in $P$. Theorem~\ref{thm:main} establishes that, for edges, the situation is simpler: the local factor equals $1$ on both sides. The bijection between the square $2$-faces of $\Ocal(P)$ and $\Ccal(P)$ constructed in the proof of \cite[Theorem~5.11]{FLM} likewise acts class by class: it applies $\Phi$ to the two difference sets of a square face of $\Ocal(P)$. Several of the elementary facts used in Section~2 (for example, that a convex set is connected if and only if its maximal and minimal elements form a connected set) also appear in \cite{FLM}; we include short proofs to keep the paper self-contained.

The rest of the paper is organized as follows. Besides the background above, this section contains two observations of our own: a characterization of the distributive lattices $L$ for which $\Ocal(L)$ and $\Ccal(L)$ are unimodularly equivalent (Proposition~\ref{prop:distributive}), and examples demonstrating that, despite Theorem~\ref{thm:HLSS}, the graphs of $\Ocal(P)$ and $\Ccal(P)$ can have different diameters, with either one the larger, by an arbitrary amount (Example~\ref{ex:diam}). Section~2 contains the proof of Theorem~\ref{thm:main}. In Section~3 we establish that the bijection obtained from Theorem~\ref{thm:main} is precisely the bijection of Hibi, Li, Sahara and Shikama~\cite{HLSS}. That bijection is given by an explicit formula; Theorem~\ref{thm:agree} reveals it to be a disjoint union of simple class-by-class bijections, which explains why it is a bijection and how it acts on edge directions and lengths.

In Section~4 we record consequences. The first group (Sections~\ref{sec:directions} and \ref{sec:opt}) uses the class-by-class form of Theorem~\ref{thm:main}, not merely the equality of totals: $\Ocal(P)$ and $\Ccal(P)$ have the same number of edge directions, with matching multiplicities (Theorem~\ref{thm:dirs}(i)); for every coordinate subspace, $\Ccal(P)$ has at least as many edges parallel to it as $\Ocal(P)$ (Theorem~\ref{thm:dirs}(ii)); and the edge lengths of $\Ccal(P)$ are dominated by those of $\Ocal(P)$ (Theorem~\ref{thm:dirs}(iii)); strict inequality occurs in the last two comparisons if and only if $P$ contains a three-element chain. Moreover, $f_1(\Ocal(Q))\le f_1(\Ocal(P))$ for every induced subposet $Q$ of $P$, convex or not (Corollary~\ref{cor:subposet}). We then interpret these results for the maximum-weight ideal and maximum-weight antichain problems and their convex multicriteria extensions, via the edge-direction framework of \citet{OR} and the edge complexity of \citet{MO} (Corollary~\ref{cor:edgecomplexity}). The second group (Sections~\ref{sec:comparability}--\ref{sec:fences}) uses (\ref{eq:formula}): a formula for $f_1$ in terms of the comparability graph, which extends to the stable-set polytope of an arbitrary graph, a simplification for posets of width at most $2$, a recursion for series-parallel posets with closed forms for layered posets and for posets with a minimum and a maximum, and closed forms for zigzag and crown posets. Finally, in Section~\ref{sec:diam} we turn to the graphs of the two polytopes: distances are bounded by numbers of components, and both diameters lie in a common interval determined by the connected components of the poset and their widths (Theorem~\ref{prop:widthdiam}); for every $m$, the two polytopes have equally many ordered pairs of vertices whose symmetric difference has $m$ components (Theorem~\ref{prop:components}), which for $m=1$ is Theorem~\ref{thm:HLSS}; and the diameters coincide for ordinal sums, and hence for series-parallel posets and for posets with a minimum or a maximum (Proposition~\ref{prop:spdiam}), and both equal the width for zigzag and crown posets (Theorem~\ref{prop:zigzagdiam}).

In Section~5 we discuss the conjecture of \citet{HL-unimod} that $f_k(\Ocal(P))\le f_k(\Ccal(P))$ for all $k$, whose inequality has recently been proved in \cite{Fourier}, and the part of it that remains open.

%=====================================================================
\section{Proof of Theorem~\ref{thm:main}}
%=====================================================================

We treat parts (i), (ii) and (iii) in turn, each preceded by the facts that its proof uses.

\subsection*{Part (i)}
We first record two simple facts. Let $I\subseteq J$ be ideals. If $a\preceq c\preceq b$ with $a,b\in J\setminus I$, then $c\in J$ because $c\preceq b$, and $c\notin I$ because $a\preceq c$ and $a\notin I$. Therefore
\begin{equation}\label{eq:diffconvex}
  J\setminus I \text{ is convex.}
\end{equation}
Next, let $S$ be convex. Then $\down S\cap\up S=S$ by convexity, and $D_S:=\down S\setminus S$ is an ideal: if $b\preceq a\in D_S\,$, then $b\in\down S$, and $b\in S$ would give $b\preceq a\preceq s$ for some $s\in S$, whence $a\in S$ by convexity. Therefore
\begin{equation}\label{eq:decomp}
  P \;=\; D_S\ \sqcup\ S\ \sqcup\ S^{\parallel}\ \sqcup\ (\up S\setminus S).
\end{equation}

\begin{lemma}\label{lem:order}
Let $S\subseteq P$ be nonempty and convex. Then $I'\mapsto (D_S\sqcup I',\ D_S\sqcup I'\sqcup S)$ is a bijection from the ideals of $S^{\parallel}$ to the pairs $(I,J)$ of ideals of $P$ with $I\subseteq J$ and $J\setminus I=S$.
\end{lemma}

\begin{proof}
Let $I'$ be an ideal of $S^{\parallel}$, let $I:=D_S\sqcup I'$, and let $b\preceq a\in I$. If $a\in D_S\,$, then $b\in D_S\,$. If $a\in I'$, then $b\notin\up S$ (otherwise $a\in\up S$); so either $b\in\down S\setminus S=D_S\,$, or $b\in S^{\parallel}$, in which case $b\in I'$. Hence $I$ is an ideal. If $a\in S$, then $b\in\down S=D_S\sqcup S$; hence $J:=I\sqcup S$ is an ideal.

Conversely, let $I\subseteq J$ be ideals with $J\setminus I=S$. Every element of $D_S$ lies below an element of $S\subseteq J$, so $D_S\subseteq J\setminus S=I$. If $s\preceq a\in I$ with $s\in S$, then $s\in I$, which is impossible; so $I\cap\up S=\emptyset$. By (\ref{eq:decomp}), $I=D_S\sqcup I'$ with $I':=I\cap S^{\parallel}$, an ideal of $S^{\parallel}$. Because $I'$ is recovered from $(I,J)$, the map is a bijection.
\end{proof}

\begin{proof}[Proof of Theorem~\ref{thm:main}(i)]
If $e\in E_{\Ocal}(S)$, then $S$ is nonempty, connected by $(\mathrm E_{\Ocal})$, and convex by (\ref{eq:diffconvex}). Conversely, if $S\in\mathfrak S(P)$, then by Theorem~\ref{thm:HL} the edges in $E_{\Ocal}(S)$ are in one-to-one correspondence with the pairs of Lemma~\ref{lem:order}, of which there are $|\mathcal J(S^{\parallel})|\ge 1$.
\end{proof}

\subsection*{Part (ii)}
Let $T\in\mathfrak T(P)$. If $|T|\ge 2$, then $T=\max T\sqcup\min T$: an element neither maximal nor minimal in $T$ would be the middle element of a three-element chain, and an element both maximal and minimal in $T$ would be incomparable to all other elements of $T$, contradicting connectivity. Moreover, if $t\prec t'$ in $T$, then $t\in\min T$ and $t'\in\max T$. A partition of $T$ into two antichains is a proper $2$-coloring of the (connected) comparability graph of $T$, and such a coloring is unique up to interchanging colors. We define
\[
  (T^+,T^-):=\begin{cases}(\max T,\min T), & |T|\ge 2;\\ (T,\emptyset), & |T|=1.\end{cases}
\]
In either case, $\{T^+,T^-\}$ is the unique unordered partition of $T$ into two antichains.

\begin{lemma}\label{lem:chain}
Let $T\in\mathfrak T(P)$. Then $C\mapsto\{C\sqcup T^+,\ C\sqcup T^-\}$ is a bijection from the antichains of $T^{\parallel}$ to the unordered pairs $\{A,B\}$ of distinct antichains of $P$ with $A\triangle B=T$.
\end{lemma}

\begin{proof}
If $C$ is an antichain of $T^{\parallel}$, then $C\sqcup T^+$ and $C\sqcup T^-$ are antichains, because every element of $C$ is incomparable to every element of $T$; their symmetric difference is $T\neq\emptyset$. Conversely, let $\{A,B\}$ be such a pair. Then $A\setminus B$ and $B\setminus A$ partition $T$ into two antichains, so $\{A\setminus B,B\setminus A\}=\{T^+,T^-\}$. Let $C:=A\cap B$. Every $c\in C$ and every $t\in T$ lie together in $A$ or together in $B$, so they are incomparable; hence $C$ is an antichain of $T^{\parallel}$. Because $C$ is recovered as $A\cap B$, the map is a bijection.
\end{proof}

\begin{proof}[Proof of Theorem~\ref{thm:main}(ii)]
If $e=[\rho(A),\rho(B)]\in E_{\Ccal}(T)$, then $T$ is nonempty and connected by $(\mathrm E_{\Ccal})$, and a chain in $T$ meets each of the antichains $A\setminus B$ and $B\setminus A$ at most once, so $T$ contains no three-element chain. Conversely, if $T\in\mathfrak T(P)$, then by Theorem~\ref{thm:HL} and Lemma~\ref{lem:chain}, $|E_{\Ccal}(T)|$ is the number of antichains of $T^{\parallel}$, namely $|\mathcal J(T^{\parallel})|\ge 1$.
\end{proof}

\subsection*{Part (iii)}
We use the following observation repeatedly: if $W\subseteq P$ is nonempty and $b\in W$, then there exist $m\in\min W$ and $M\in\max W$ with $m\preceq b\preceq M$, and we may take $m=b$ (resp., $M=b$) when $b\in\min W$ (resp., $b\in\max W$).

\begin{proof}[Proof of Theorem~\ref{thm:main}(iii)]
Let $S\in\mathfrak S(P)$, and let $T:=\Phi(S)$. Being the union of two antichains, $T$ contains no three-element chain. To see that $T$ is connected, let $t,t'\in T$, and let $t=s_0\,,s_1\,,\dots,s_k=t'$ be elements of $S$ with $s_{j-1}$ and $s_j$ comparable for $1\le j\le k$. Choose $m_j\in\min S$ and $M_j\in\max S$ with $m_j\preceq s_j\preceq M_j\,$, $t\in\{m_0\,,M_0\}$ and $t'\in\{m_k\,,M_k\}$. Then $m_j\preceq M_j$ for all $j$; moreover, $s_{j-1}\preceq s_j$ implies $m_{j-1}\preceq M_j\,$, and $s_{j-1}\succeq s_j$ implies $M_{j-1}\succeq m_j\,$. Hence $t$ and $t'$ are joined within $T$, and $T\in\mathfrak T(P)$.

Conversely, let $T\in\mathfrak T(P)$, and let $S:=\cl{T}$. Clearly $S$ is convex, and $S$ is connected because each of its elements is comparable to an element of the connected set $T\subseteq S$. We claim that $\max S=\max T$. If $a\in\max T$ and $a\prec c\in S$, then $c\preceq t''$ for some $t''\in\max T$, so $a\prec t''$, contradicting that $\max T$ is an antichain; thus $\max T\subseteq\max S$. If $c\in\max S$, then $c\preceq t''$ for some $t''\in\max T\subseteq S$, so $c=t''$; thus $\max S\subseteq\max T$. Similarly, $\min S=\min T$. Therefore $\Phi(\cl{T})=\max T\cup\min T=T$. In the other direction, for $S\in\mathfrak S(P)$, convexity gives $\cl{\Phi(S)}\subseteq S$, and the observation above gives $S\subseteq\cl{\Phi(S)}$.

Finally, $\Phi(S)\subseteq S$ gives $S^{\parallel}\subseteq\Phi(S)^{\parallel}$. If $c$ is comparable to some $s\in S$, say $c\succeq s$ (resp., $c\preceq s$), then $c\succeq m$ for some $m\in\min S$ (resp., $c\preceq M$ for some $M\in\max S$), so $c\notin\Phi(S)^{\parallel}$. Hence $\Phi(S)^{\parallel}=S^{\parallel}$.
\end{proof}

\begin{example}\label{ex:X}
For the poset $X$ of Figure~\ref{fig:X}, a set $S\in\mathfrak S(X)$ either contains $c$, in which case it is $\{c\}$ together with an arbitrary subset of $\{a,b,g,h\}$ and $S^{\parallel}=\emptyset$; or it is one of $\{a\},\{b\},\{g\},\{h\}$, because every comparable pair of elements other than $c$ has $c$ strictly between them, and then $S^{\parallel}$ is a single point. Thus $|\mathfrak S(X)|=20$, and (\ref{eq:formula}) gives $f_1(\Ocal(X))=f_1(\Ccal(X))=16\cdot1+4\cdot2=24$, in agreement with the degree sequences of the two graphs in \cite[Example~1.2]{HLSS}, each of which sums to $48$ (see also Corollary~\ref{cor:layered}).
\end{example}

%=====================================================================
\section{The bijection of Hibi, Li, Sahara and Shikama}
%=====================================================================

Hibi, Li, Sahara and Shikama~\cite{HLSS} proved Theorem~\ref{thm:HLSS} by writing down an explicit map from the edges of $\Ocal(P)$ to the edges of $\Ccal(P)$ and verifying directly that it is a bijection. The formula, recalled below, does not reveal why the map is a bijection, nor how it acts on the directions and lengths of edges. In this section we establish that this map is exactly the bijection $\Psi$ obtained from Theorem~\ref{thm:main}. Thus the decomposition of Theorem~\ref{thm:main} is not a competing construction, but the structure underlying the bijection of \cite{HLSS}.

For $S\in\mathfrak S(P)$, let $\minc S:=\min S$ if $|S|\ge2$, and $\minc S:=\emptyset$ if $|S|=1$. Note that $(\Phi(S)^+,\Phi(S)^-)=(\max S,\minc S)$. Combining Lemmas~\ref{lem:order} and \ref{lem:chain} with the bijection $I'\mapsto\max I'$ between the ideals and the antichains of $S^{\parallel}$, we obtain the bijection
\[
  \Psi:\ \bigl[\rho(D_S\sqcup I'),\ \rho(D_S\sqcup I'\sqcup S)\bigr]
  \;\longmapsto\;
  \bigl[\rho(\max I'\cup\minc S),\ \rho(\max I'\cup\max S)\bigr]
\]
from the edges of $\Ocal(P)$ to the edges of $\Ccal(P)$, which carries $E_{\Ocal}(S)$ onto $E_{\Ccal}(\Phi(S))$. In \cite{HLSS}, an edge $[\rho(I),\rho(J)]$ of $\Ocal(P)$ with $I\subseteq J$ is mapped to $[\rho(A),\rho(B)]$, where
\[
  A:=\max J,\qquad B:=\minc(J\setminus I)\cup(\max I\cap\max J).
\]

\begin{theorem}\label{thm:agree}
The bijection $\Psi$ coincides with the bijection of Hibi, Li, Sahara and Shikama~\cite{HLSS}.
\end{theorem}

\begin{proof}
Let $I=D_S\sqcup I'$ and $J=I\sqcup S$. It suffices to establish that $\max J=\max I'\cup\max S$ and that $\max I\cap\max J=\max I'$. No element of $D_S$ is maximal in $J$. Let $s\in S$ and $s\prec c\in J$. Then $c\notin I'$, because $I'\subseteq S^{\parallel}$; and $c\notin D_S\,$, because $c\preceq s'\in S$ would give $c\in S$ by convexity. Hence $c\in S$, and so $s\in\max J$ if and only if $s\in\max S$. Let $a\in I'$ and $a\prec c\in J$. Then $c\notin S\cup D_S\,$, because otherwise $a$ would be comparable to an element of $S$; hence $c\in I'$, and so $a\in\max J$ if and only if $a\in\max I'$. Therefore $\max J=\max I'\cup\max S$, and $\max J\cap I=\max I'$. Because an element of $I$ that is maximal in $J$ is maximal in $I$, we have $\max I\cap\max J=\max J\cap I=\max I'$.
\end{proof}

Note that $\Psi$ maps the vertex $\rho(D_S\sqcup I')$ to $\rho(\max I'\cup\minc S)$, which depends on $S$ and not only on the vertex. This is consistent with \cite[Theorem~3.1]{HLSS}: when $X$ appears in $P$, the graphs of $\Ocal(P)$ and $\Ccal(P)$ have different degree sequences, so no bijection between the vertex sets preserves adjacency.

Theorem~\ref{thm:agree} explains the bijection of Hibi, Li, Sahara and Shikama~\cite{HLSS}. It is a disjoint union of bijections $E_{\Ocal}(S)\to E_{\Ccal}(\Phi(S))$, one for each $S\in\mathfrak S(P)$, and on each class it is simply the bijection $I'\mapsto\max I'$ from the ideals of $S^{\parallel}$ to its antichains; this is why it is a bijection. Consequently, the bijection of \cite{HLSS} maps each parallel class of edges of $\Ocal(P)$ onto a parallel class of edges of $\Ccal(P)$ and never increases the length of an edge (Theorem~\ref{thm:dirs} below); as noted above, it is not induced by a map on vertices. None of the later results depends on Theorem~\ref{thm:agree}, but it demonstrates that they are properties of the bijection of \cite{HLSS} itself.

%=====================================================================
\section{Consequences}
%=====================================================================

This section collects consequences of Theorem~\ref{thm:main} and of formula (\ref{eq:formula}). We have the following highlights.
\begin{itemize}
\item Matched classes of edges are parallel classes, so the two polytopes have equally many edge directions, with matching multiplicities; the chain polytope has at least as many edges parallel to every coordinate subspace, and its edge lengths are dominated by those of the order polytope (Theorem~\ref{thm:dirs}).
\item The number of edges of the order polytope does not decrease under passing to an arbitrary induced subposet, although the order polytope of a subposet need not be a face (Corollary~\ref{cor:subposet}).
\item Both diameters lie in a common interval determined by the components of the poset and their widths, and they coincide for ordinal sums, series-parallel posets, and zigzag and crown posets (Theorem~\ref{prop:widthdiam}(ii), Proposition~\ref{prop:spdiam} and Theorem~\ref{prop:zigzagdiam}).
\item For every $m$, the two polytopes have equally many pairs of vertices whose symmetric difference has $m$ components; $m=1$ gives Theorem~\ref{thm:HLSS} (Theorem~\ref{prop:components}).
\end{itemize}
The section is organized as follows. Section~\ref{sec:directions} treats edge directions, edge lengths and subposets, and Section~\ref{sec:opt} interprets the results for linear and convex combinatorial optimization over ideals and antichains, including edge complexity (Corollary~\ref{cor:edgecomplexity}). Sections~\ref{sec:comparability}--\ref{sec:fences} use (\ref{eq:formula}) to count edges: in terms of the comparability graph, for posets of width at most $2$, for series-parallel, layered and bounded posets, and for zigzag and crown posets. Section~\ref{sec:diam} turns from edges to the graphs of the two polytopes: distances and diameters (Theorems~\ref{prop:widthdiam} and \ref{prop:zigzagdiam}, Proposition~\ref{prop:spdiam}) and the component refinement of Theorem~\ref{thm:HLSS} (Theorem~\ref{prop:components}).

\subsection{Edge directions and edge lengths}\label{sec:directions}
Theorem~\ref{thm:HLSS} compares only the total numbers of edges. Because the bijection of Theorem~\ref{thm:main} respects difference sets, it also compares how the edges lie in $\RR^d$. For an ideal $I'$ of $S^{\parallel}$, the edge of $\Ocal(P)$ in $E_{\Ocal}(S)$ joining $\rho(D_S\sqcup I')$ and $\rho(D_S\sqcup I'\sqcup S)$ has direction $\rho(S)$; for an antichain $C$ of $T^{\parallel}$, the edge of $\Ccal(P)$ in $E_{\Ccal}(T)$ joining $\rho(C\sqcup T^-)$ and $\rho(C\sqcup T^+)$ has direction $\rho(T^+)-\rho(T^-)$. In each case the direction depends only on the difference set, and the difference set is recovered from the direction as its support. Hence the parallel classes of edges of $\Ocal(P)$ are exactly the sets $E_{\Ocal}(S)$, $S\in\mathfrak S(P)$, and those of $\Ccal(P)$ are exactly the sets $E_{\Ccal}(T)$, $T\in\mathfrak T(P)$. Theorem~\ref{thm:main} therefore gives part (i) of the following theorem. For $U\subseteq P$, we say that an edge is \emph{parallel to} $\RR^U$ if its direction lies in the coordinate subspace spanned by $\{\mathbf e_i: a_i\in U\}$; equivalently, its difference set is contained in $U$. The squared Euclidean length of an edge with difference set $S$ (for $\Ocal(P)$) or $T$ (for $\Ccal(P)$) is $|S|$ (resp., $|T|$). Because $\Phi(S)\subseteq S$, with equality if and only if $S$ contains no three-element chain, the bijection $\Psi$ never increases length.

\begin{theorem}\label{thm:dirs}
Let $P$ be a finite poset.
\begin{enumerate}[(iii)]
\renewcommand{\theenumi}{\roman{enumi}}\renewcommand{\labelenumi}{\textup{(\theenumi)}}
\item The polytopes $\Ocal(P)$ and $\Ccal(P)$ have the same number, $|\mathfrak S(P)|$, of edge directions, and $\Phi$ matches the parallel classes of edges of $\Ocal(P)$ with those of $\Ccal(P)$ so that matched classes have the same size, $|\mathcal J(S^{\parallel})|$.
\item For every $U\subseteq P$, the number of edges of $\Ocal(P)$ parallel to $\RR^U$ is at most the number of edges of $\Ccal(P)$ parallel to $\RR^U$. Equality holds for every $U$ if and only if $P$ contains no three-element chain.
\item For every $t\ge1$, the number of edges of $\Ccal(P)$ of length at least $\sqrt t$ is at most the number of edges of $\Ocal(P)$ of length at least $\sqrt t$. In particular,
\[
  \sum_{e} \|e\|^2 \;=\; \sum_{S\in\mathfrak S(P)} |\mathcal J(S^{\parallel})|\,\bigl|\max S\cup\min S\bigr|
  \;\le\; \sum_{S\in\mathfrak S(P)} |\mathcal J(S^{\parallel})|\,|S| \;=\; \sum_{e'} \|e'\|^2,
\]
where $e$ ranges over the edges of $\Ccal(P)$, $e'$ over the edges of $\Ocal(P)$, and $\|e\|$ denotes the length of $e$. More generally, $\sum_e g(\|e\|)\le\sum_{e'}g(\|e'\|)$ for every nondecreasing function $g$; if $g$ is strictly increasing, the inequality is strict if and only if $P$ contains a three-element chain.
\end{enumerate}
\end{theorem}

\begin{proof}
Part (i) follows from Theorem~\ref{thm:main} and the description of the parallel classes above.

For (ii), the two numbers are $\sum_{S\in\mathfrak S(P),\,S\subseteq U}|\mathcal J(S^{\parallel})|$ and $\sum_{T\in\mathfrak T(P),\,T\subseteq U}|\mathcal J(T^{\parallel})|$. If $S\subseteq U$, then $\Phi(S)\subseteq S\subseteq U$; so $\Phi$ maps $\{S\in\mathfrak S(P): S\subseteq U\}$ injectively into $\{T\in\mathfrak T(P): T\subseteq U\}$, and $|\mathcal J(\Phi(S)^{\parallel})|=|\mathcal J(S^{\parallel})|$. This proves the inequality. If $P$ contains no three-element chain, then every $S\in\mathfrak S(P)$ satisfies $\Phi(S)=S$, and equality holds for every $U$. If $a\prec b\prec c$ in $P$, let $U:=\{a,c\}$ and $T:=U\in\mathfrak T(P)$. Then $b\in\cl{T}\setminus U$, so $T$ is not in the image of $\{S\in\mathfrak S(P): S\subseteq U\}$, and because $|\mathcal J(T^{\parallel})|\ge1$ the inequality is strict.

For (iii), the inequalities follow from the fact that $\Psi$ is a bijection that does not increase length. Let $g$ be strictly increasing. If $P$ contains no three-element chain, then $\Phi$ is the identity and $\Psi$ preserves length. If $a\prec b\prec c$ in $P$, then the interval $S:=P[a,c]$ belongs to $\mathfrak S(P)$ and contains $b\notin\max S\cup\min S$, so $|\Phi(S)|<|S|$, and the $|\mathcal J(S^{\parallel})|\ge1$ edges in $E_{\Ocal}(S)$ are mapped to strictly shorter edges.
\end{proof}

The one-sided statement that the edge directions of $\Ocal(P)$ are exactly the vectors $\rho(S)$, $S\in\mathfrak S(P)$, follows directly from Theorem~\ref{thm:HL}, and the analogous statement for $\Ccal(P)$ follows from Theorem~\ref{thm:HL} and Lemma~\ref{lem:chain}; what is new in Theorem~\ref{thm:dirs}(i) is the matching of the two sets of directions, with equal multiplicities. In the language of linear optimization, the edge directions of a polytope are the possible pivot directions of the simplex method over it, so the two polytopes offer equally many pivot directions, with matching multiplicities. Parts (ii) and (iii) exploit the inclusion $\Phi(S)\subseteq S$, which says that $\Phi$ can only shrink supports. By (ii), although the two polytopes have the same number of edges, the chain polytope has at least as many edges as the order polytope parallel to each coordinate subspace; in the example $U=\{a,c\}$ with $a\prec b\prec c$, the edges of $\Ccal(P)$ parallel to $\RR^U$ with difference set $\{a,c\}$ correspond, under $\Psi^{-1}$, to edges of $\Ocal(P)$ whose difference set is the interval $P[a,c]$, which is not contained in $U$.

For the poset $X$ (Example~\ref{ex:X}), both polytopes have $20$ edge directions, and the sums of the squared lengths of their $24$ edges are $56$ for $\Ocal(X)$ and $47$ for $\Ccal(X)$.

In augmentation and pivoting algorithms over $0/1$-polytopes, the $\ell_1$-length of an edge, which equals its squared Euclidean length, is the number of variables that change in the corresponding step; in these terms, Theorem~\ref{thm:dirs}(iii) says that the steps available over antichains are no longer than the matched steps over ideals, and Theorem~\ref{thm:dirs}(ii) says that, for every set $U$ of variables allowed to change, there are at least as many steps over antichains as over ideals.

\begin{remark}\label{rem:limitations}
Theorem~\ref{thm:dirs} has two limitations. First, parts (ii) and (iii) concern the given coordinates rather than the combinatorial type: for a chain $P$, the polytopes $\Ocal(P)$ and $\Ccal(P)$ are both simplices, yet the inequalities are strict once $|P|\ge3$. Second, when $P$ has no three-element chain, all three parts of Theorem~\ref{thm:dirs} can also be seen directly, because $\Ocal(P)$ is then mapped onto $\Ccal(P)$ by an isometry of the cube, namely $x_i\mapsto 1-x_i$ on the coordinates of the minimal elements of $P$, which carries $\rho(I)$ to $\rho(I\triangle\min P)$. Theorem~\ref{thm:dirs} thus has content only in the presence of three-element chains, where no such map exists. Applied to the face of $\Ccal(P)$ formed by the antichains $C$ with $A\cap B\subseteq C\subseteq A\cup B$, for two antichains $A,B$, which is a copy of the chain polytope of the subposet $A\triangle B$ of height at most $2$, this isometry gives the case of comparability graphs of a theorem of \citet[Theorem~3.4]{IT}, which states that such faces of stable-set polytopes are equivalent to order polytopes.
\end{remark}

The class-by-class form of Theorem~\ref{thm:main} also transfers monotonicity under passing to subposets from the chain polytope, where it is clear, to the order polytope. For an induced subposet $Q$ of $P$ and $W\subseteq Q$, we write $W^{\parallel_Q}$ and $W^{\parallel_P}$ for the sets of elements of $Q$, respectively of $P$, that are incomparable to every element of $W$.

\begin{corollary}\label{cor:subposet}
Let $Q$ be an induced subposet of $P$. Then $S\mapsto\cl{\Phi_Q(S)}$, with the convex closure taken in $P$, is an injective map from $\mathfrak S(Q)$ to $\mathfrak S(P)$, and $|\mathcal J(S^{\parallel_Q})|\le|\mathcal J(\cl{\Phi_Q(S)}^{\parallel_P})|$ for every $S\in\mathfrak S(Q)$. Consequently, $|\mathfrak S(Q)|\le|\mathfrak S(P)|$ and $f_1(\Ocal(Q))\le f_1(\Ocal(P))$, with strict inequalities if $Q\neq P$.
\end{corollary}

\begin{proof}
Let $S\in\mathfrak S(Q)$ and $T:=\Phi_Q(S)\in\mathfrak T(Q)$. Comparability in $Q$ is comparability in $P$, so $T$ is connected in $P$ and contains no three-element chain of $P$; thus $T\in\mathfrak T(P)$, and $\cl T\in\mathfrak S(P)$ with $\Phi_P(\cl T)=T$ by Theorem~\ref{thm:main}(iii). As $\Phi_Q$ is injective, so is $S\mapsto\cl T$. By Theorem~\ref{thm:main}(iii), applied in $Q$ and in $P$, $S^{\parallel_Q}=T^{\parallel_Q}=T^{\parallel_P}\cap Q$ and $\cl T^{\parallel_P}=T^{\parallel_P}$. The antichains of the induced subposet $T^{\parallel_P}\cap Q$ are antichains of $T^{\parallel_P}$, which gives the inequality of class sizes, and the two consequences follow from (\ref{eq:formula}). If $p\in P\setminus Q$, then $\{p\}\in\mathfrak S(P)$ is not in the image, because $\Phi_P(\{p\})=\{p\}\not\subseteq Q$; this gives strictness.
\end{proof}

On the chain side, Corollary~\ref{cor:subposet} is clear, because $\Ccal(Q)$ is the face of $\Ccal(P)$ on which the coordinates outside $Q$ vanish. On the order side it is not: $\Ocal(Q)$ is a face of $\Ocal(P)$ only when $Q$ is convex, and a connected convex subset of $Q$ need not be convex in $P$.

\subsection{Ramifications for combinatorial optimization}\label{sec:opt}
\begin{figure}[ht]
\centering
\begin{tikzpicture}[x={(2.2cm,-0.45cm)},y={(-1.1cm,0.75cm)},z={(0cm,2.2cm)},
  vtx/.style={circle,fill=black,inner sep=1.3pt},
  cA/.style={very thick,black},
  cB/.style={very thick,blue!70!black,dashed},
  cC/.style={very thick,red!70!black,dotted},
  cD/.style={very thick,green!50!black,dash dot},
  cE/.style={very thick,orange!85!black,densely dashed},
  cF/.style={line width=2.2pt,violet}]
  % axes
  \draw[->,gray!50] (0,0,0)--(1.35,0,0) node[gray,font=\scriptsize,right]{$x_1$};\draw[->,gray!50] (0,0,0)--(0,1.35,0) node[gray,font=\scriptsize,left]{$x_2$};\draw[->,gray!50] (0,0,0)--(0,0,1.3) node[gray,font=\scriptsize,above]{$x_3$};
  % order polytope: ideals 0,1,12,123
  \draw[cA] (0,0,0)--(1,0,0);
  \draw[cB] (1,0,0)--(1,1,0);
  \draw[cC] (1,1,0)--(1,1,1);
  \draw[cD] (0,0,0)--(1,1,0);
  \draw[cE] (1,0,0)--(1,1,1);
  \draw[cF] (0,0,0)--(1,1,1);
  \node[vtx,label=below:{\scriptsize$\emptyset$}] at (0,0,0) {};
  \node[vtx,label=below:{\scriptsize$\{1\}$}] at (1,0,0) {};
  \node[vtx,label=right:{\scriptsize$\{1,2\}$}] at (1,1,0) {};
  \node[vtx,label=right:{\scriptsize$\{1,2,3\}$}] at (1,1,1) {};
  \node at (0.55,0,-0.4) {$\Ocal(P)$};
\end{tikzpicture}
\qquad
\begin{tikzpicture}[x={(2.2cm,-0.45cm)},y={(-1.1cm,0.75cm)},z={(0cm,2.2cm)},
  vtx/.style={circle,fill=black,inner sep=1.3pt},
  cA/.style={very thick,black},
  cB/.style={very thick,blue!70!black,dashed},
  cC/.style={very thick,red!70!black,dotted},
  cD/.style={very thick,green!50!black,dash dot},
  cE/.style={very thick,orange!85!black,densely dashed},
  cF/.style={line width=2.2pt,violet}]
  \draw[->,gray!50] (0,0,0)--(1.35,0,0) node[gray,font=\scriptsize,right]{$x_1$};\draw[->,gray!50] (0,0,0)--(0,1.35,0) node[gray,font=\scriptsize,left]{$x_2$};\draw[->,gray!50] (0,0,0)--(0,0,1.3) node[gray,font=\scriptsize,above]{$x_3$};
  % chain polytope: antichains 0,1,2,3
  \draw[cA] (0,0,0)--(1,0,0);
  \draw[cB] (0,0,0)--(0,1,0);
  \draw[cC] (0,0,0)--(0,0,1);
  \draw[cD] (1,0,0)--(0,1,0);
  \draw[cE] (0,1,0)--(0,0,1);
  \draw[cF] (1,0,0)--(0,0,1);
  \node[vtx,label=below:{\scriptsize$\emptyset$}] at (0,0,0) {};
  \node[vtx,label=below:{\scriptsize$\{1\}$}] at (1,0,0) {};
  \node[vtx,label=left:{\scriptsize$\{2\}$}] at (0,1,0) {};
  \node[vtx,label=left:{\scriptsize$\{3\}$}] at (0,0,1) {};
  \node at (0.55,0,-0.4) {$\Ccal(P)$};
\end{tikzpicture}

\medskip
\small
\begin{tabular}{@{}lllcc@{}}
\toprule
$S\in\mathfrak S(P)$ & edge of $\Ocal(P)$ & edge of $\Ccal(P)$ & $\|e'\|^2$ & $\|e\|^2$\\
\midrule
$\{1\}$ & $\emptyset$ -- $\{1\}$ & $\emptyset$ -- $\{1\}$ & 1 & 1\\
$\{2\}$ & $\{1\}$ -- $\{1,2\}$ & $\emptyset$ -- $\{2\}$ & 1 & 1\\
$\{3\}$ & $\{1,2\}$ -- $\{1,2,3\}$ & $\emptyset$ -- $\{3\}$ & 1 & 1\\
$\{1,2\}$ & $\emptyset$ -- $\{1,2\}$ & $\{1\}$ -- $\{2\}$ & 2 & 2\\
$\{2,3\}$ & $\{1\}$ -- $\{1,2,3\}$ & $\{2\}$ -- $\{3\}$ & 2 & 2\\
$\{1,2,3\}$ & $\emptyset$ -- $\{1,2,3\}$ & $\{1\}$ -- $\{3\}$ & 3 & 2\\
\bottomrule
\end{tabular}
\caption{The bijection $\Psi$ for the three-element chain $P=\{a_1\prec a_2\prec a_3\}$, with $a_i$ abbreviated to $i$ in the labels and corresponding to the coordinate $x_i\,$. Vertices are labelled by ideals (left) and antichains (right); matched edges are drawn in the same style. Each polytope is a tetrahedron with six edge directions, each of multiplicity one (Theorem~\ref{thm:dirs}(i)). The long diagonal of $\Ocal(P)$, with difference set $P[1,3]=\{1,2,3\}$, is matched with the shorter edge $[\rho(\{1\}),\rho(\{3\})]$ of $\Ccal(P)$, so the sums of squared lengths are $10>9$ (Theorem~\ref{thm:dirs}(iii)). Exactly two edges of $\Ocal(P)$ but three edges of $\Ccal(P)$ are parallel to the coordinate plane $\RR^{\{1,3\}}$ (Theorem~\ref{thm:dirs}(ii)). The edge complexities are $\operatorname{ec}(\Ocal(P))=3$ and $\operatorname{ec}(\Ccal(P))=2$ (Corollary~\ref{cor:edgecomplexity}).}\label{fig:chain}
\end{figure}

Linear optimization over $\Ocal(P)$ is the maximum-weight ideal problem, equivalently the maximum-weight closure problem, which reduces to a single minimum-cut computation (see \cite{Picard}). Linear optimization over $\Ccal(P)$ is the maximum-weight antichain problem, which is solvable in polynomial time because $\Ccal(P)$ is the stable-set polytope of the perfect graph $G(P)$ (see \cite{GLS}). We record what the preceding results say about these two classical problems and their nonlinear extensions; Figure~\ref{fig:chain} illustrates the smallest poset, the three-element chain, for which the comparisons of Theorem~\ref{thm:dirs}(ii) and (iii) are strict.

\emph{Local optimality.} By Theorem~\ref{thm:HL}, an ideal $I$ is a maximum-weight ideal if and only if no move ``add or remove a set $S\in\mathfrak S(P)$'' leads to a better ideal, and an antichain $A$ is a maximum-weight antichain if and only if no move ``exchange the two layers of a set $T\in\mathfrak T(P)$'' leads to a better antichain; Theorem~\ref{thm:main} adds that the two families of moves are in canonical bijection, and that, for each $S\in\mathfrak S(P)$, the moves along $S$ and along $\Phi(S)$ are available at the same number, $2\,|\mathcal J(S^{\parallel})|$, of vertices.

\emph{Convex combinatorial optimization.} Let $M$ be an integer $r\times P$ matrix and $f:\RR^r\to\RR$ a convex function, and consider the problems
\[
  \max\{f(M\rho(I)):\ I \text{ an ideal of } P\}\quad\text{and}\quad \max\{f(M\rho(A)):\ A \text{ an antichain of } P\}.
\]
Every such problem attains its maximum at a solution whose image is a vertex of the projected polytope $M\Ocal(P)$ (resp., $M\Ccal(P)$). \citet{OR} established that, if a set $E$ of $\mu$ vectors containing a direction of every edge of the underlying $0/1$-polytope is available, then such a problem reduces to linear optimization over the same family for each of at most $2\sum_{k=0}^{r-1}\binom{\mu-1}{k}$ objectives. The reason is that every edge of the projected polytope is parallel to $Mh$ for some $h\in E$, so the \emph{zonotope} $\sum_{h\in E}[-Mh,Mh]$ refines the projected polytope, in the sense that the normal cone of each of its vertices is contained in the normal cone of a vertex of the projected polytope; each of its at most $2\sum_{k=0}^{r-1}\binom{\mu-1}{k}$ vertices supplies one linear objective (see \cite[Lemmas~2.1 and 2.2, and the proof of Theorem~2.3]{MO}). By Theorem~\ref{thm:dirs}(i), the sets
\[
  \{\rho(S):\ S\in\mathfrak S(P)\}\qquad\text{and}\qquad \{\rho(\max S)-\rho(\minc S):\ S\in\mathfrak S(P)\}
\]
are complete sets of edge directions for $\Ocal(P)$ and $\Ccal(P)$, of the same size $\mu=|\mathfrak S(P)|$. Hence the bound on the number of linear counterparts is the same for convex optimization over ideals and over antichains, and both sets of directions are produced, element by element, from the single list $\mathfrak S(P)$. The bound is useful when $|\mathfrak S(P)|$ is small compared with the number $|\mathcal J(P)|$ of feasible solutions. For the zigzag poset $Z_n\,$, with elements $a_1\,,\dots,a_n$ and cover relations $a_1\prec a_2\succ a_3\prec a_4\succ \cdots$ (Section~\ref{sec:fences}), we have $|\mathfrak S(Z_n)|=\binom{n+1}{2}$ (the segments), whereas $|\mathcal J(Z_n)|=F_{n+2}$; so, for fixed $r$, convex $r$-criteria optimization over the ideals, or over the antichains, of $Z_n$ reduces to $O(n^{2(r-1)})$ linear counterparts, although there are exponentially many feasible solutions. For instance, for $n=40$ and two criteria, at most $2\mu=1{,}640$ linear counterparts suffice, whereas $Z_{40}$ has $F_{42}=267{,}914{,}296$ ideals.

\emph{Edge complexity.} \citet{MO} define the \emph{edge complexity} $\operatorname{ec}(\mathcal Q)$ of a lattice polytope $\mathcal Q$ as the least nonnegative integer such that every edge of $\mathcal Q$ is parallel to an integer vector of $\ell_1$-norm at most $\operatorname{ec}(\mathcal Q)$, and they establish that, for fixed $r$, the projection of $\mathcal Q\subseteq\RR^n$ by an integer $r\times n$ matrix $M$ has $O\bigl((\operatorname{ec}(\mathcal Q)\,\|M\|_\infty)^{r(r-1)}\bigr)$ vertices, and that, given a linear-optimization oracle and a bound on $\operatorname{ec}(\mathcal Q)$, convex maximization of $f(Mx)$ over the vertices of $\mathcal Q$ reduces to that many linear counterparts (see \cite[Theorem~2.3 and its proof]{MO}). For a $0/1$-polytope, the difference of the endpoints of an edge is a $\{0,\pm1\}$-vector, which is the shortest integer vector in its direction; hence $\operatorname{ec}(\mathcal Q)$ is the largest cardinality of the support of such a difference.

\begin{corollary}\label{cor:edgecomplexity}
Let $P$ be nonempty. The edge complexity of $\Ocal(P)$ is the largest cardinality of a connected component of $P$. The edge complexity of $\Ccal(P)$ is
\[
  \operatorname{ec}(\Ccal(P))=\max_{S\in\mathfrak S(P)}\bigl|\max S\cup\min S\bigr|\;\le\;\min\{\operatorname{ec}(\Ocal(P)),\,2w(P)\},
\]
and $\operatorname{ec}(\Ccal(P))=\operatorname{ec}(\Ocal(P))$ if and only if some connected component of $P$ of largest cardinality contains no three-element chain.
\end{corollary}

\begin{proof}
The edges of $\Ocal(P)$ in $E_{\Ocal}(S)$ have endpoint difference $\rho(S)$, of $\ell_1$-norm $|S|$. Every $S\in\mathfrak S(P)$ lies in a connected component of $P$, and every connected component is itself in $\mathfrak S(P)$; this gives $\operatorname{ec}(\Ocal(P))$. The edges of $\Ccal(P)$ in $E_{\Ccal}(T)$ have endpoint difference $\rho(T^+)-\rho(T^-)$, of $\ell_1$-norm $|T|$; by Theorem~\ref{thm:main}(iii), $\{|T|: T\in\mathfrak T(P)\}=\{|\Phi(S)|: S\in\mathfrak S(P)\}$, which gives the formula. Because $\Phi(S)\subseteq S$, we have $|\Phi(S)|\le|S|\le \operatorname{ec}(\Ocal(P))$; and because $\Phi(S)$ is the union of two antichains, $|\Phi(S)|\le 2w(P)$. Finally, $|\Phi(S)|=\operatorname{ec}(\Ocal(P))$ forces $|S|=\operatorname{ec}(\Ocal(P))$ and $\Phi(S)=S$, that is, $S$ is a connected component of largest cardinality containing no three-element chain; conversely, such a component $S$ has $|\Phi(S)|=|S|$.
\end{proof}

Combined with \cite[Theorem~2.3]{MO}, Corollary~\ref{cor:edgecomplexity} gives: for fixed $r$, $\beta$ and $w$, there is a constant $\kappa(r,\beta,w)$ such that, for every poset $P$ of width at most $w$ and every $\{0,1,\dots,\beta\}$-valued $r\times P$ matrix $M$, the projected polytope $M\Ccal(P)$ has at most $\kappa(r,\beta,w)$ vertices, and the convex $r$-criteria maximum-weight antichain problem reduces to at most $\kappa(r,\beta,w)$ linear maximum-weight antichain problems, independently of $|P|$. The corresponding bound for ideals obtained in this way depends on the largest connected component of $P$, which is not bounded in terms of the width; for a chain, $\operatorname{ec}(\Ocal(P))=|P|$ whereas $\operatorname{ec}(\Ccal(P))=\min\{|P|,2\}$.

\subsection{The comparability graph}\label{sec:comparability}
For a graph $G$ and a set $U$ of vertices, $G[U]$ is the subgraph induced by $U$, $N[U]$ is the closed neighborhood of $U$, $G-N[U]$ is the subgraph induced by the vertices not in $N[U]$, and $\ind(G)$ is the number of stable sets of $G$ (including $\emptyset$).

Because $\Ccal(P)$ is the stable-set polytope of $G(P)$, posets with isomorphic comparability graphs have chain polytopes that coincide up to a permutation of coordinates; \citet{Stanley} used this to establish that the number of linear extensions is a \emph{comparability invariant}, that is, depends only on $G(P)$ up to isomorphism. Combined with Theorem~\ref{thm:HLSS}, this establishes at once that $f_1(\Ocal(P))$ is a comparability invariant; Corollary~\ref{cor:comparability} makes this explicit.

\begin{corollary}\label{cor:comparability}
Let $P$ be a finite poset with comparability graph $G$. Then
\[
  f_1(\Ocal(P))=f_1(\Ccal(P))=\sum_{T}\ind\bigl(G-N[T]\bigr),
\]
where the sum is over all nonempty $T\subseteq P$ such that $G[T]$ is connected and triangle-free.
\end{corollary}

\begin{proof}
A set $T\subseteq P$ is connected if and only if $G[T]$ is connected. Because three pairwise comparable elements form a chain, $T$ contains a three-element chain if and only if $G[T]$ contains a triangle. The vertex set of $G-N[T]$ is $T^{\parallel}$, and the antichains of $T^{\parallel}$ are the stable sets of $G-N[T]$. The result now follows from Theorem~\ref{thm:main}(ii) and (\ref{eq:formula}).
\end{proof}

The same argument, starting from the adjacency criterion of \citet{Chvatal}, establishes that for an arbitrary graph $G$ the stable-set polytope of $G$ has $\sum_T \ind(G-N[T])$ edges, the sum being over all nonempty $T$ such that $G[T]$ is connected and bipartite. For comparability graphs, an induced subgraph is bipartite if and only if it is triangle-free, because a triangle-free comparability graph comes from a poset with no three-element chain.

For posets of width at most $2$, formula (\ref{eq:formula}) simplifies further.

\begin{corollary}\label{cor:width2}
If $P$ has width at most $2$, then
\[
  f_1(\Ocal(P))=f_1(\Ccal(P))=|\mathfrak S(P)|+\sum_{S\in\mathfrak S(P)}|S^{\parallel}|.
\]
\end{corollary}

\begin{proof}
For $S\in\mathfrak S(P)$, the set $S^{\parallel}$ is a chain, because two incomparable elements of $S^{\parallel}$ together with an element of $S$ would form a three-element antichain. Hence $|\mathcal J(S^{\parallel})|=|S^{\parallel}|+1$, and the claim follows from (\ref{eq:formula}).
\end{proof}

Posets of width $2$ need not be $X$-free ($X$ itself has width $2$), so Corollary~\ref{cor:width2} is not a consequence of Theorem~\ref{thm:X}.

\begin{remark}\label{rem:XY}
The comparability invariance of $f_0(\Ocal(P))$ and $f_1(\Ocal(P))$ is not shared by $f_2(\Ocal(P))$ or $f_{d-1}(\Ocal(P))$, although, by \cite[Theorem~5.11]{FLM}, it is shared by the number of square $2$-faces. With $\oplus$ denoting ordinal sum (Section~\ref{sec:sp}), consider
\[
  X=A_2\oplus A_1\oplus A_2\qquad\text{and}\qquad Y:=A_2\oplus A_2\oplus A_1\,,
\]
which have the same comparability graph, $K_{2,2,1}\,$. No element of $Y$ has two incomparable elements above it and two incomparable elements below it, so $Y$ is $X$-free. By Theorem~\ref{thm:X}, $\Ocal(Y)$ is unimodularly equivalent to $\Ccal(Y)$, which coincides with $\Ccal(X)$ up to a permutation of coordinates. Hence $\Ocal(Y)$ has the face numbers of $\Ccal(X)$, and, by \cite[Theorem~6.1]{FLM}, $f_2(\Ocal(X))<f_2(\Ocal(Y))$; and, by \cite{HL-unimod}, $f_{d-1}(\Ocal(X))<f_{d-1}(\Ocal(Y))$. (Directly: $\Ocal(X)$ has $4+2+2=8$ facets and $\Ocal(Y)$ has $6+2+1=9$, counting cover relations, minimal elements and maximal elements; see \cite{Stanley}.) Moreover, by \cite[Example~1.2]{HLSS}, the graphs of $\Ocal(X)$ and $\Ocal(Y)\cong\Ccal(X)$ have different degree sequences $(6,6,6,6,6,6,6,6)$ and $(5,6,6,6,6,6,6,7)$, whereas both have $24$ edges.
\end{remark}

\subsection{Series-parallel posets}\label{sec:sp}
Let $P+Q$ denote the disjoint union of the posets $P$ and $Q$, and let $P\oplus Q$ denote their ordinal sum, in which every element of $P$ is below every element of $Q$. The \emph{series-parallel} posets are those obtained from the one-element poset by repeated application of $+$ and $\oplus$; they are precisely the posets that do not contain the four-element poset N, with elements $a_1\,,a_2\,,a_3\,,a_4$ and cover relations $a_1\prec a_2\succ a_3\prec a_4\,$, as an induced subposet (see \cite[Exercise~3.14(a)]{EC1} and \cite{VTL}). In what follows, $f_1(R)$ denotes the common value of $f_1(\Ocal(R))$ and $f_1(\Ccal(R))$ for a finite poset $R$. \citet{FL} express the $f$-polynomials of the order and chain polytopes of $P+Q$ and of $P\oplus Q$ in terms of those of $P$ and $Q$, using that the polytopes of $P+Q$ are products and that those of $P\oplus Q$ are subdirect sums in the sense of McMullen. Extracting the coefficient of edges gives the following proposition; for completeness, we derive it directly from (\ref{eq:formula}).

\begin{proposition}\label{prop:sp}
For finite posets $P$ and $Q$,
\begin{align*}
  |\mathcal J(P+Q)|&=|\mathcal J(P)|\,|\mathcal J(Q)|, & f_1(P+Q)&=f_1(P)\,|\mathcal J(Q)|+|\mathcal J(P)|\,f_1(Q),\\
  |\mathcal J(P\oplus Q)|&=|\mathcal J(P)|+|\mathcal J(Q)|-1, & f_1(P\oplus Q)&=f_1(P)+f_1(Q)\\
  &&&\quad+\bigl(|\mathcal J(P)|-1\bigr)\bigl(|\mathcal J(Q)|-1\bigr).
\end{align*}
\end{proposition}

\begin{proof}
The formula for $|\mathcal J(P+Q)|$ follows from $\mathcal J(P+Q)\cong \mathcal J(P)\times \mathcal J(Q)$ (see \cite[Section~3.4]{EC1}), and the formula for $|\mathcal J(P\oplus Q)|$ holds because the ideals of $P\oplus Q$ are the ideals of $P$ and the sets $P\cup J$ with $J$ a nonempty ideal of $Q$. A connected subset of $P+Q$ lies in $P$ or in $Q$; for $S\subseteq P$, convexity is as in $P$, and $S^{\parallel}$ is the disjoint union of the set $S^{\parallel}$ computed in $P$ and all of $Q$. Because $|\mathcal J(\cdot)|$ is multiplicative over disjoint unions, the first formula for $f_1$ follows from (\ref{eq:formula}).

For $P\oplus Q$, there are three cases to consider.

\begin{case}
$S\subseteq P$. Convexity and connectivity of $S$, and the set $S^{\parallel}$, are the same in $P\oplus Q$ as in $P$, because every element of $Q$ is comparable to every element of $S$. These sets contribute $f_1(P)$.
\end{case}

\begin{case}
$S\subseteq Q$. Similarly, these sets contribute $f_1(Q)$.
\end{case}

\begin{case}
$S$ meets both $P$ and $Q$. Then $S$ is connected, and $S^{\parallel}=\emptyset$, so each such $S$ contributes $1$. Moreover, $S$ is convex if and only if $S\cap P$ is an up-set of $P$ and $S\cap Q$ is an ideal of $Q$. Indeed, if $S$ is convex, $a\in S\cap P$ and $a\preceq c\in P$, then $a\preceq c\preceq b$ for any $b\in S\cap Q$, so $c\in S$; similarly, $S\cap Q$ is an ideal. Conversely, let $a\preceq c\preceq b$ with $a,b\in S$; if $c\in P$, then $a\in P$, so $c\in S$; if $c\in Q$, then $b\in Q$, so $c\in S$. Because $P$ has $|\mathcal J(P)|-1$ nonempty up-sets (the complements of its proper ideals) and $Q$ has $|\mathcal J(Q)|-1$ nonempty ideals, there are $(|\mathcal J(P)|-1)(|\mathcal J(Q)|-1)$ such sets.
\end{case}
\noindent Summing the three contributions gives the second formula for $f_1\,$.
\end{proof}

Let $u(P):=|\mathcal J(P)|-1$ and $h(P):=f_1(P)-\binom{u(P)}{2}$. By Proposition~\ref{prop:sp}, $|\mathcal J(P)|+f_1(P)\,\varepsilon$ is multiplicative under $+$ in the ring $\mathbb Z[\varepsilon]/(\varepsilon^2)$, and $u$ and $h$ are additive under $\oplus$, because $\binom{u+v}{2}=\binom u2+\binom v2+uv$. For example, $|\mathcal J(A_n)|+f_1(A_n)\varepsilon=(2+\varepsilon)^n$, so $f_1(A_n)=n2^{n-1}$, the number of edges of the $n$-cube. Although it can be extracted from the ordinal-sum formulas of \cite{FL}, we have not found the following explicit formula stated in the literature.

\begin{corollary}\label{cor:layered}
For $P=A_{n_1}\oplus\cdots\oplus A_{n_k}\,$,
\[
  f_1(\Ocal(P))=f_1(\Ccal(P))
  =\binom{\sum_{j=1}^k (2^{n_j}-1)}{2}+\sum_{j=1}^k\left(n_j2^{\,n_j-1}-\binom{2^{n_j}-1}{2}\right).
\]
\end{corollary}

\begin{proof}
We have $u(A_n)=2^n-1$ and $h(A_n)=n2^{n-1}-\binom{2^n-1}{2}$, and $f_1=\binom u2+h$.
\end{proof}

For a $k$-element chain, Corollary~\ref{cor:layered} gives $\binom{k+1}{2}$, the number of edges of a $k$-simplex. The formula is symmetric in $n_1\,,\dots,n_k\,$, in accordance with Corollary~\ref{cor:comparability}, because permuting the layers does not change the comparability graph $K_{n_1\,,\dots,n_k}$; by Remark~\ref{rem:XY}, permuting the layers can change the combinatorial type of $\Ocal(P)$.

\begin{corollary}\label{cor:bounded}
If $\hat R$ is obtained from a finite poset $R$ by adjoining a new minimum and a new maximum, then
\[
  f_1(\Ocal(\hat R))=f_1(\Ocal(R))+2\,|\mathcal J(R)|+1.
\]
\end{corollary}

\begin{proof}
Because $\hat R=A_1\oplus R\oplus A_1\,$, $|\mathcal J(A_1)|=2$ and $f_1(A_1)=1$, two applications of Proposition~\ref{prop:sp} give $|\mathcal J(R\oplus A_1)|=|\mathcal J(R)|+1$, $f_1(R\oplus A_1)=f_1(R)+|\mathcal J(R)|$, and then $f_1(\hat R)=1+f_1(R)+|\mathcal J(R)|+|\mathcal J(R)|=f_1(R)+2\,|\mathcal J(R)|+1$.
\end{proof}

In particular, the number of edges of the order polytope of a finite lattice with at least two elements is determined by its proper part, that is, by the subposet obtained by removing its minimum and its maximum.

\subsection{Zigzag and crown posets}\label{sec:fences}
For $n\ge1$, the \emph{zigzag poset} (or \emph{fence poset}) $Z_n$ has elements $a_1\,,\dots,a_n$ and cover relations $a_1\prec a_2\succ a_3\prec a_4\succ\cdots$ (see \cite[Exercise~3.66]{EC1}); its comparability graph is the path $a_1a_2\cdots a_n\,$. For $n\ge2$, the \emph{crown poset} $\mathrm{Cr}_n$ has elements $a_1\,,\dots,a_n\,,b_1\,,\dots,b_n$ and cover relations $a_i\prec b_i$ and $a_{i+1}\prec b_i$ for $1\le i\le n$, with $a_{n+1}:=a_1$; its comparability graph is the cycle $a_1b_1a_2b_2\cdots a_nb_n$ of length $2n$, so $\mathrm{Cr}_n$ is the closed analogue of $Z_{2n}\,$. Neither poset contains a three-element chain, so all of their subsets are convex; the connected subsets of $Z_n$ are the segments of the path, and those of $\mathrm{Cr}_n$ are the arcs of the cycle and the whole cycle. The order polytope of $Z_n$ has been studied extensively from the point of view of Ehrhart theory (see \cite{PZ} and the references there), and order polytopes of crown posets are studied in \cite{LL}. For $n\ge4$, $\{a_1\,,a_2\,,a_3\,,a_4\}$ is an induced copy of N in $Z_n\,$, and for $n\ge3$, $\{a_1\,,b_1\,,a_2\,,b_2\}$ is an induced copy of N in $\mathrm{Cr}_n\,$, so Proposition~\ref{prop:sp} does not apply. We let $F_k$ and $L_k$ denote the Fibonacci and Lucas numbers, with $F_1=F_2=1$, $L_1=1$ and $L_2=3$.

\begin{proposition}\label{prop:fence}
For $n\ge1$, $f_0(\Ocal(Z_n))=F_{n+2}$ and
\[
  f_1(\Ocal(Z_n))=f_1(\Ccal(Z_n))=\sum_{\substack{s,t\ge1\\ s+t\le n+1}}F_sF_t
  = 1+\frac{n\,L_{n+3}-F_{n+4}-2F_{n+1}}{5}.
\]
For $n\ge2$, $f_0(\Ocal(\mathrm{Cr}_n))=L_{2n}$ and
\[
  f_1(\Ocal(\mathrm{Cr}_n))=f_1(\Ccal(\mathrm{Cr}_n))=1+2n\,(F_{2n+1}-1).
\]
\end{proposition}

\begin{proof}
A path on $k\ge 0$ vertices has $F_{k+2}$ stable sets; hence $f_0(\Ocal(Z_n))=F_{n+2}\,$. The connected subsets of $Z_n$ are the segments $S=\{a_j\,,a_{j+1}\,,\dots,a_k\}$, $1\le j\le k\le n$. For such $S$, $S^{\parallel}=\{a_l:l\le j-2\}\sqcup\{a_l:l\ge k+2\}$ is a disjoint union of two sets, with no comparabilities between them, each of whose comparability graphs is a path (on $\max\{j-2,0\}$ and $\max\{n-k-1,0\}$ vertices), so $|\mathcal J(S^{\parallel})|=F_jF_{n+1-k}\,$. Setting $s=j$ and $t=n+1-k$ in (\ref{eq:formula}) gives the sum.

Let $\Sigma_n$ denote the sum, and let $R_n$ denote the right-hand side. Then $\Sigma_1=R_1=1$. Comparing coefficients in $\bigl(q/(1-q-q^2)\bigr)^2$, one obtains the classical identity $\sum_{k=0}^m F_kF_{m-k}=(mL_m-F_m)/5$, whence $5(\Sigma_n-\Sigma_{n-1})=(n+1)L_{n+1}-F_{n+1}\,$. On the other hand,
\[
  5(R_n-R_{n-1})=nL_{n+1}+L_{n+2}-F_{n+2}-2F_{n-1}=(n+1)L_{n+1}-F_{n+1}\,,
\]
the last equality because $L_{n+2}-L_{n+1}-(F_{n+2}-F_{n+1})=L_n-F_n=2F_{n-1}\,$. The result follows by induction.

For $\mathrm{Cr}_n\,$, a cycle of length $m$ has $L_m$ stable sets, so $f_0(\Ocal(\mathrm{Cr}_n))=L_{2n}\,$. The whole cycle $S$ has $S^{\parallel}=\emptyset$ and contributes $1$ to (\ref{eq:formula}). For each $1\le\ell\le 2n-1$, there are $2n$ arcs $S$ of length $\ell$, and $S^{\parallel}$ consists of the vertices of the cycle that are neither in $S$ nor adjacent to $S$; it induces a path on $\max\{2n-\ell-2,0\}$ vertices, so $|\mathcal J(S^{\parallel})|=F_{2n-\ell}\,$. Hence $f_1=1+2n\sum_{m=1}^{2n-1}F_m=1+2n\,(F_{2n+1}-1)$.
\end{proof}

The faces of order polytopes can also be enumerated through connected compatible partitions (see \cite{Stanley}), and \citet{LL} use this method to give explicit formulas for all face numbers of $\Ocal(\mathrm{Cr}_n)$. In particular, $f_0(\Ocal(\mathrm{Cr}_n))=L_{2n}$ is \cite[Corollary~3.7]{LL}, and our closed form for $f_1(\Ocal(\mathrm{Cr}_n))$ agrees with the case $k=1$ of \cite[Theorem~3.6]{LL}, where it is expressed as a double sum. We have not found the formula for $f_1(\Ocal(Z_n))$ in Proposition~\ref{prop:fence} stated in the literature. Because $F_k=\varphi^k/\sqrt5+O(\varphi^{-k})$ and $L_k=\varphi^k+O(\varphi^{-k})$, where $\varphi=(1+\sqrt5)/2$, the average degree of the graph of $\Ocal(Z_n)$ is $2f_1/f_0=(1+1/\sqrt5)\,n+O(1)$.

\subsection{Distances and diameters}\label{sec:diam}
Theorem~\ref{thm:main} compares the edges of $\Ocal(P)$ and $\Ccal(P)$ class by class, but not the way in which the edges are attached to the vertices. The edges of $\Ocal(P)$ with difference set $S$ join the vertices $\rho(D_S\sqcup I')$ and $\rho(D_S\sqcup I'\sqcup S)$, for the ideals $I'$ of $S^{\parallel}$, whereas the matched edges of $\Ccal(P)$ join $\rho(C\sqcup\Phi(S)^-)$ and $\rho(C\sqcup\Phi(S)^+)$, for the antichains $C$ of $S^{\parallel}$. For instance, the vertex $\rho(\emptyset)$ of $\Ocal(P)$ lies on one edge for each nonempty connected ideal of $P$, whereas the vertex $\rho(\emptyset)$ of $\Ccal(P)$ lies only on the $d$ edges whose difference sets are singletons. This redistribution of edges among the vertices is what changes the degree sequences when $X$ appears in $P$ (Section~3), and it can also change the diameters (Example~\ref{ex:diam}). We now establish that the class structure nevertheless controls the diameters to a large extent. The key point is that, on both sides, the vertices can be joined by paths whose steps have the connected components of the symmetric difference as their difference sets; differences between the two graphs arise from shortcuts that use elements outside the symmetric difference. The same component structure also yields a refinement of Theorem~\ref{thm:HLSS}: for every $m$, the two polytopes have equally many ordered pairs of vertices whose symmetric difference has $m$ components (Theorem~\ref{prop:components}).

For $D\subseteq P$, let $\nu(D)$ denote the number of connected components of $D$.

\begin{theorem}\label{prop:widthdiam}
Let $P$ be a finite poset.
\begin{enumerate}[(iii)]
\renewcommand{\theenumi}{\roman{enumi}}\renewcommand{\labelenumi}{\textup{(\theenumi)}}
\item For ideals $I,J$ of $P$, the distance between $\rho(I)$ and $\rho(J)$ in the graph of $\Ocal(P)$ is at most $\nu(I\triangle J)$; for antichains $A,B$ of $P$, the distance between $\rho(A)$ and $\rho(B)$ in the graph of $\Ccal(P)$ is at most $\nu(A\triangle B)$.
\item Both diameters lie in a common interval:
\[
  \sum_K\delta(K)\ \le\ \operatorname{diam}\Ocal(P),\ \operatorname{diam}\Ccal(P)\ \le\ w(P)=\sum_K w(K),
\]
where both sums are over the connected components $K$ of $P$, and $\delta(K)$ is $1$ if $K$ is a chain and $2$ otherwise. Consequently, $|\operatorname{diam}\Ocal(P)-\operatorname{diam}\Ccal(P)|\le\sum_K\bigl(w(K)-\delta(K)\bigr)$; in particular, the two diameters agree whenever every connected component of $P$ has width at most $2$.
\end{enumerate}
\end{theorem}

\begin{proof}
For (i), no element of $I\setminus J$ is comparable to an element of $J\setminus I$: if $a\in I\setminus J$ and $b\in J\setminus I$, then $a\prec b$ would give $a\in J$, and $b\prec a$ would give $b\in I$. Hence the components of $I\triangle J$ are the components of $I\setminus J$ together with those of $J\setminus I$. Let $I'$ be an ideal with $I\cap J\subseteq I'\subseteq I$ that is a union of $I\cap J$ and some components of $I\setminus J$, and let $K$ be one of these components. If $x\in K$ and $x\prec y\in I'$, then $y\notin J$, so $y\in I\setminus J$ is comparable to $x$ and lies in $K$; hence $I'\setminus K$ is an ideal, and $I'$ and $I'\setminus K$ are adjacent by Theorem~\ref{thm:HL}, because $K$ is connected. Removing the components of $I\setminus J$ one at a time thus leads from $I$ to $I\cap J$, and symmetrically, adding the components of $J\setminus I$ one at a time leads from $I\cap J$ to $J$. For antichains, if $K$ is a component of $A\triangle B$, then no element of $K$ is comparable to an element of $A\setminus K$, because such an element would lie in $A\cap B$ or in another component; hence $A\triangle K$ is an antichain adjacent to $A$, and exchanging the components of $A\triangle B$ one at a time leads from $A$ to $B$.

For (ii), elements of distinct components are incomparable, so choosing one element from each component of a set $D$ gives an antichain; hence $\nu(D)\le w(P)$, and (i) bounds both diameters by $w(P)$. Also $w(P)=\sum_K w(K)$, because an antichain of $P$ is a union of antichains of its components. For the lower bound, both diameters are additive over disjoint unions (Example~\ref{ex:diam}). If $K$ is a chain, then $\Ocal(K)$ and $\Ccal(K)$ are simplices, of diameter $1$. Otherwise, $K$ has incomparable elements $a,b$; then $\{a\}$ and $\{b\}$ are not adjacent in $\Ccal(K)$, and $\down a$ and $\down b$ are not adjacent in $\Ocal(K)$, because neither contains the other; so both diameters are at least $2$.
\end{proof}

For the chain polytope, the bound $\operatorname{diam}\Ccal(P)\le w(P)$ is also a special case of \cite[Proposition~16]{ABBLS}.

We now verify the diameters claimed in Example~\ref{ex:diam}. In each case, Theorem~\ref{prop:widthdiam} supplies the upper bound, and the lower bound is established directly.

\begin{proposition}\label{prop:ex5}
Let $P_1$ consist of $X$ together with elements $p,q$ with covers $a\prec p\prec g$ and $a\prec q\prec h$, and let $P_2$ consist of $X$ together with elements $p,q,r,s$ with covers $a\prec p\prec g$, $q\prec h$, $q\prec r$, $b\prec r$ and $s\prec r$.
\begin{enumerate}[(ii)]
\renewcommand{\theenumi}{\roman{enumi}}\renewcommand{\labelenumi}{\textup{(\theenumi)}}
\item $\operatorname{diam}\Ocal(P_1)=2$ and $\operatorname{diam}\Ccal(P_1)=3$.
\item $\operatorname{diam}\Ocal(P_2)=4$ and $\operatorname{diam}\Ccal(P_2)=3$.
\end{enumerate}
\end{proposition}

\begin{proof}[Proof of Proposition~\ref{prop:ex5}(i)]
Write $P=P_1$. The chains $\{a,p,g\}$, $\{b,c,h\}$ and $\{q\}$ cover $P$, so $w(P)=3$ and $\operatorname{diam}\Ccal(P)\le3$; the argument in Example~\ref{ex:diam} gives equality. Every element other than $b$ lies above $a$. An ideal $I$ with $b\notin I$ is contained in $\{a,p,q\}$, so it is empty or connected, and $P\setminus I\supseteq\{b,c,g,h\}$ is connected because $b\prec c\prec g,h$, $p\prec g$ and $q\prec h$; hence $\rho(I)$ is equal or adjacent to both $\rho(\emptyset)$ and $\rho(P)$. An ideal $J\ne\{b\}$ with $b\in J$ contains $a$, so $J\setminus\{b\}$ is connected and $\rho(J)$ is adjacent to $\rho(\{b\})$; moreover, $J$ is connected if $c\in J$, and $P\setminus J\supseteq\{c,g,h\}$ is connected if $c\notin J$, so $\rho(J)$ is adjacent to $\rho(\emptyset)$ or to $\rho(P)$, as is $\rho(\{b\})$. Hence two ideals not containing $b$ are joined through $\rho(\emptyset)$, two containing $b$ through $\rho(\{b\})$, and one of each through $\rho(\emptyset)$ or $\rho(P)$, so $\operatorname{diam}\Ocal(P)=2$.
\end{proof}

\begin{proof}[Proof of Proposition~\ref{prop:ex5}(ii) for $\Ocal(P_2)$]
Write $P=P_2$. The elements comparable to $r$ are $b,q,s$, and $s$ is comparable only to $r$. The chains $\{a,p,g\}$, $\{b,c,h\}$, $\{q,r\}$ and $\{s\}$ cover $P$, and $\{a,b,q,s\}$ is an antichain, so $w(P)=4$ and $\operatorname{diam}\Ocal(P)\le4$ by Theorem~\ref{prop:widthdiam}(ii).

For the lower bound, let $I=\{a,b,c\}$ and $J=\{a,b,p,q,s\}$. A nonempty up-set of $I$ contains $c$, and within $P\setminus I=\{g,h,p,q,r,s\}$ the only comparabilities are $p\prec g$, $q\prec h$, $q\prec r$ and $s\prec r$, so a connected subset of $P\setminus I$ meeting two of $p,q,s$ contains $q,r,s$ and not $p$. Hence every ideal $K$ with $\rho(K)$ equal or adjacent to $\rho(I)$ satisfies one of: ($\alpha$)~$K\subseteq\{a,b\}$; ($\beta$)~$c\in K$ and $|K\cap\{p,q,s\}|\le1$; ($\gamma$)~$c\in K$, $\{q,r,s\}\subseteq K$ and $p\notin K$. Within $J$ the only comparability is $a\prec p$, and within $P\setminus J=\{c,g,h,r\}$ the only ones are $c\prec g$ and $c\prec h$; so a connected set removed from $J$ is a singleton or $\{a,p\}$, and a connected set added to $J$ is $\{r\}$ or contains $c$. Hence every ideal $K'$ with $\rho(K')$ equal or adjacent to $\rho(J)$ satisfies one of: ($\alpha'$)~$c,r\notin K'$ and $|K'\cap\{p,q,s\}|\ge2$; ($\beta'$)~$K'=J\cup\{r\}$; ($\gamma'$)~$J\cup\{c\}\subseteq K'$ and $r\notin K'$. Such $K$ and $K'$ are never equal, and we claim that they are never adjacent, so that the distance between $\rho(I)$ and $\rho(J)$ is at least $4$. Suppose they were adjacent, with difference set $D$. If $K$ satisfies ($\alpha$), then $K\subset K'$ and $D$ contains two of $p,q,s$; in case ($\alpha'$), $D\subseteq\{a,b,p,q,s\}$, where the only comparability is $a\prec p$; in case ($\beta'$), the component of $p$ in $D$ lies in $\{a,p\}$; in case ($\gamma'$), $s\in D$ and $r\notin D$, so $s$ is isolated in $D$. If $K$ satisfies ($\beta$) or ($\gamma$) and $K'$ satisfies ($\alpha'$) or ($\beta'$), then $c\in K\setminus K'$, so $K'\subset K$; this excludes ($\beta$), and in case ($\gamma$) it forces $p\notin K'$, so $K'$ is $J\setminus\{p\}$ or $J\setminus\{a,p\}$ and contains $b,q,s$, whence $r$ is isolated in $D\ni c,r$. If $K$ satisfies ($\beta$) or ($\gamma$) and $K'$ satisfies ($\gamma'$), then $K'\not\subseteq K$ because $p,q,s\in K'$, so $K\subset K'$; this excludes ($\gamma$) because $r\in K\setminus K'$, and in case ($\beta$) either $s\in D$ and $s$ is isolated in $D$, or $s\in K$, $p,q\in D\subseteq\{g,h,p,q\}$, and $p,q$ lie in different components of $D$. Therefore $\operatorname{diam}\Ocal(P)=4$.
\end{proof}

\begin{proof}[Proof of Proposition~\ref{prop:ex5}(ii) for $\Ccal(P_2)$]
Write $P=P_2$. As in the previous proof, the elements comparable to $r$ are $b,q,s$, the element $s$ is comparable only to $r$, and the chains $\{a,p,g\}$, $\{b,c,h\}$, $\{q,r\}$ and $\{s\}$ cover $P$. The antichains adjacent to $\emptyset$ are the singletons, and no singleton $\{x\}$ is adjacent to $\{a,b,s\}$: if $x\in\{a,b,s\}$, the other two elements are incomparable, and otherwise $x$ would be comparable to $s$, so $x=r$, but $r$ is incomparable to $a$. Thus $\operatorname{diam}\Ccal(P)\ge3$. For the upper bound, let $A,B$ be antichains; by Theorem~\ref{prop:widthdiam}(i) we may assume $\nu(A\triangle B)=4$. Choosing one element in each component gives an antichain of size $4$, which meets each of the four chains above; as $r$ is comparable to $s$, $h$ to $q$, $g$ to $b$ and $c$, and $a$ to $c$, it is one of $\{a,b,q,s\}$, $\{b,p,q,s\}$ and $\{c,p,q,s\}$. So $q$ and $s$ lie in different components of $A\triangle B$; hence $r\notin A\triangle B$, and $r\notin A\cap B$ because $r$ is comparable to $q$. Likewise $g$ (comparable to $a,b,c,p$) and $h$ (comparable to $a,b,c,q$) would join two components of $A\triangle B$, and cannot lie in $A\cap B$, whose elements are incomparable to every element of $A\triangle B$. In $\{a,b,c,p,q,s\}$ the only comparabilities are $a\prec c$, $b\prec c$ and $a\prec p$, so $A\cap B=\emptyset$ (each of $a,b,c,p$ in $A\cap B$ would leave at most one component of $A\triangle B$ in $\{a,b,c,p\}$), and $A\cup B=V\cup\{q,s\}$ with $V\in\bigl\{\{a,b\},\{a,b,p\},\{b,p\},\{b,c,p\},\{c,p\}\bigr\}$, the subsets of $\{a,b,c,p\}$ with two components.

If $V$ is $\{a,b\}$, $\{a,b,p\}$ or $\{b,p\}$, let $Y\in\{A,B\}$ contain $b$ and let $Z$ be the other. Then $X:=(Y\setminus\{b,q,s\})\cup\{r\}$ is an antichain, $Y\triangle X=\{r\}\cup(Y\cap\{b,q,s\})$ is connected, and $X\triangle Z=(V\setminus\{b\})\cup(Z\cap\{q,s\})\cup\{r\}$ has two components, so $A$ and $B$ are at distance at most $3$.

If $c\in V$, let $W\in\{A,B\}$ contain $c$, let $Z$ be the other, and let $H=\{g,h,r\}$; an antichain disjoint from $H$ is adjacent to $H$ when its union with $H$ is connected. If $q\in W$, then $W\cup H$ is connected because $c$ is comparable to $g,h$ and $q$ to $h,r$; and $Z\subseteq\{b,p,s\}$, so either $b\in Z$ and $Z\cup H$ is connected, or $H$, $Z\cup\{b\}$, $Z$ is a path. If $q\in Z$ and $b\in Z$, then $Z\cup H$ is connected, and $W\subseteq\{c,p,s\}$, $(W\setminus\{s\})\cup\{r\}$, $H$ is a path. If $q\in Z$ and $b\notin V$, then $W\subseteq\{c,p,s\}$, $Z\subseteq\{p,q,s\}$, and
\[
  W,\quad (W\setminus\{c,p\})\cup\{a,b\},\quad (Z\setminus\{q,s\})\cup\{h,r\},\quad Z
\]
is a path, the three difference sets being $\{a,b,c\}\cup(W\cap\{p\})$, $\{a,b,h,r\}\cup(W\cap\{s\})\cup(Z\cap\{p\})$ and $\{h,r\}\cup(Z\cap\{q,s\})$. In every case $A$ and $B$ are at distance at most $3$, so $\operatorname{diam}\Ccal(P)=3$.
\end{proof}

For the first poset of Example~\ref{ex:diam}, the interval is $[2,3]$, with $\operatorname{diam}\Ocal(P)=2$ and $\operatorname{diam}\Ccal(P)=3$. Both bounds of Theorem~\ref{prop:widthdiam}(ii) are attained: the lower bound by every series-parallel poset (Proposition~\ref{prop:spdiam}), and the upper bound by every zigzag poset and every crown poset (Theorem~\ref{prop:zigzagdiam}). Series-parallel posets need not be $X$-free; nevertheless, by Proposition~\ref{prop:spdiam} and Theorem~\ref{prop:widthdiam}(ii), a difference of diameters as in Example~\ref{ex:diam} requires an induced copy of N and a connected component of width at least $3$.

Theorem~\ref{thm:HLSS} counts the pairs of adjacent vertices, which by Theorem~\ref{thm:HL} are the pairs whose symmetric difference is connected. The class structure of Theorem~\ref{thm:main} extends this count to every number of components.

\begin{theorem}\label{prop:components}
For every $m\ge0$, the number of ordered pairs $(I,J)$ of ideals of $P$ with $\nu(I\triangle J)=m$ equals the number of ordered pairs $(A,B)$ of antichains of $P$ with $\nu(A\triangle B)=m$. Equivalently,
\[
  \sum_{I,J}q^{\,\nu(I\triangle J)}=\sum_{A,B}q^{\,\nu(A\triangle B)},
\]
where the sums are over all ideals $I,J$ and all antichains $A,B$ of $P$.
\end{theorem}

\begin{proof}
For a set $W\subseteq P$, let $\mathcal K(W)$ be the set of components of $W$. Elements of different components are incomparable, and the components of a convex set $S$ are convex: if $a\preceq c\preceq b$ with $a,b$ in a component $K$ of $S$, then $c\in S$ is comparable to $a$, so $c\in K$.

\emph{Ideals.} Let $N_m$ be the number of pairs $(I_1,I_2)$ of ideals with $I_1\subseteq I_2$ and $\nu(I_2\setminus I_1)=m$. Given ideals $I,J$, the pair $(I\cap J,I\cup J)$ is such a pair with $(I\cup J)\setminus(I\cap J)=I\triangle J$, and, as in the proof of Theorem~\ref{prop:widthdiam}, $\mathcal K(I\triangle J)=\mathcal K(I\setminus J)\sqcup\mathcal K(J\setminus I)$. Conversely, given $(I_1,I_2)$ and any of the $2^m$ ways of distributing the components of $I_2\setminus I_1$ between two sets $Q_1$ and $Q_2\,$, the sets $I_1\cup Q_1$ and $I_1\cup Q_2$ are ideals, as in that proof, with intersection $I_1$ and union $I_2\,$. So the number of ordered pairs of ideals with $\nu(I\triangle J)=m$ is $2^mN_m\,$. By Lemma~\ref{lem:order}, $N_m=\sum_S|\mathcal J(S^{\parallel})|$, where the sum is over the convex sets $S$ with $\nu(S)=m$; for $m=0$ this holds trivially, with $S=\emptyset$ and $N_0=|\mathcal J(P)|$.

\emph{Antichains.} Given antichains $A,B$, the set $T:=A\triangle B$ contains no three-element chain, $C:=A\cap B$ is an antichain of $T^{\parallel}$, and each component of $T$ meets $A$ in one of its two layers, as in Part~(ii) of Section~2. Conversely, given a set $T$ with no three-element chain and $\nu(T)=m$, an antichain $C$ of $T^{\parallel}$, and one of the $2^m$ choices of a layer of each component of $T$, the set $A$ consisting of $C$ and the chosen layers and the set $B$ consisting of $C$ and the other layers are antichains with $A\triangle B=T$, because elements of different components are incomparable. So the number of ordered pairs of antichains with $\nu(A\triangle B)=m$ is $2^m\sum_T|\mathcal J(T^{\parallel})|$, where the sum is over the sets $T$ with no three-element chain and $\nu(T)=m$.

\emph{Matching.} Finally, $S\mapsto\bigcup_{K\in\mathcal K(S)}\Phi(K)$ is a bijection from the convex sets to the sets with no three-element chain that preserves the number of components and the set of incomparable elements. Indeed, the sets $\Phi(K)$, $K\in\mathcal K(S)$, lie in the mutually incomparable components $K$, so they are the components of the image, and the image has incomparable set $\bigcap_K\Phi(K)^{\parallel}=\bigcap_K K^{\parallel}=S^{\parallel}$ by Theorem~\ref{thm:main}(iii). In the other direction, if $T$ has no three-element chain, then the convex closures of its components are mutually incomparable: if $x\in\cl{T_1}$ and $y\in\cl{T_2}$ with $x\preceq y$, then $t_1\preceq x\preceq y\preceq t_2$ for some $t_1\in T_1$ and $t_2\in T_2\,$, so $T_1$ and $T_2$ are the same component. Hence their union is convex, with these closures as its components, and it is mapped to $T$. So the two sums agree.
\end{proof}

For $m=0$, Theorem~\ref{prop:components} says that $\Ocal(P)$ and $\Ccal(P)$ have equally many vertices, and for $m=1$ it is Theorem~\ref{thm:HLSS}, with each edge counted twice. By Theorem~\ref{prop:widthdiam}(i), the pairs counted for a given $m$ are joined by paths of length at most $m$ in both graphs.

\begin{proposition}\label{prop:spdiam}
Let $P=P_1\oplus P_2$ with $P_1\,,P_2$ nonempty. Then $\operatorname{diam}\Ocal(P)=\operatorname{diam}\Ccal(P)$, and this common value is $1$ if $P$ is a chain and $2$ otherwise. Consequently, $\operatorname{diam}\Ocal(P)=\operatorname{diam}\Ccal(P)=\sum_K\delta(K)$ for every series-parallel poset $P$, with the sum and $\delta$ as in Theorem~\ref{prop:widthdiam}(ii), and both diameters are at most $2$ for every poset with a minimum or a maximum, in particular for every finite lattice.
\end{proposition}

\begin{proof}
If $P$ is a chain, both polytopes are simplices, of diameter $1$. Otherwise, both diameters are at least $2$, as in the proof of Theorem~\ref{prop:widthdiam}, and it remains to bound them by $2$.

The antichains of $P$ are those of $P_1$ and those of $P_2\,$. If $A\subseteq P_1$ and $B\subseteq P_2$ are nonempty antichains, then $A\triangle B=A\cup B$ is connected, because every element of $A$ is below every element of $B$; so $A$ and $B$ are adjacent. Any two antichains of the same $P_i\,$, including $\emptyset$, are both adjacent to $\{c\}$ for any $c$ in the other part, by the same argument.

The ideals of $P$ are the ideals $I$ of $P_1$ and the sets $P_1\cup J$ with $J$ a nonempty ideal of $P_2\,$. Every ideal $I\subsetneq P_1$ is adjacent to $P$, because $P\setminus I=(P_1\setminus I)\cup P_2$ is connected; every ideal $P_1\cup J$ with $J\neq\emptyset$ is adjacent to $\emptyset$, because $P_1\cup J$ is connected; and an ideal $I\subsetneq P_1$ is adjacent to an ideal $P_1\cup J$ with $J\neq\emptyset$, because $(P_1\setminus I)\cup J$ is connected. Finally, $P_1$ is adjacent to $P_1\setminus\{p_1\}$ and to $P_1\cup\{p_2\}$ for any $p_1\in\max P_1$ and $p_2\in\min P_2\,$, and these are adjacent to every ideal $P_1\cup J$ with $J\neq\emptyset$ and to every ideal $I\subsetneq P_1\,$, respectively, because $\{p_1\}\cup J$ and $(P_1\setminus I)\cup\{p_2\}$ are connected. Hence any two ideals are at distance at most $2$.

For the consequences, a connected series-parallel poset with at least two elements is an ordinal sum of two nonempty posets, so the first statement and additivity over disjoint unions give the formula; and a poset with a minimum or a maximum, and at least two elements, is of the form $A_1\oplus R$ or $R\oplus A_1$ with $R$ nonempty.
\end{proof}

For zigzag and crown posets, the upper bound is attained, and the bound on distances in Theorem~\ref{prop:widthdiam}(i) is exact.

\begin{theorem}\label{prop:zigzagdiam}
Let $P$ be $Z_n$ with $n\ge1$, or $\mathrm{Cr}_n$ with $n\ge2$. Then the distance between antichains $A$ and $B$ in the graph of $\Ccal(P)$ is $\nu(A\triangle B)$. Consequently, $\operatorname{diam}\Ocal(P)=\operatorname{diam}\Ccal(P)=w(P)$, which is $\lceil n/2\rceil$ for $Z_n$ and $n$ for $\mathrm{Cr}_n\,$.
\end{theorem}

\begin{proof}
By Theorem~\ref{prop:widthdiam}(i), the distance between $A$ and $B$ is at most $\nu(A\triangle B)$. Conversely, along an edge $[\rho(A'),\rho(A'')]$, the symmetric difference with $B$ changes from $D:=A'\triangle B$ to $D\triangle T$, where $T:=A'\triangle A''$ is connected by Theorem~\ref{thm:HL}. We claim that $\nu(D\triangle T)\ge \nu(D)-1$; because $\nu(B\triangle B)=0$, it then follows that every path from $A$ to $B$ has length at least $\nu(A\triangle B)$.

If $T$ is the whole cycle of $\mathrm{Cr}_n\,$, then $D\triangle T$ is the complement of $D$, which has as many components as $D$ unless $D$ is empty or the whole cycle, and the claim is clear. Otherwise, $T$ is a segment or an arc $t_1t_2\cdots t_r$ of the comparability graph, with neighbors $t_0$ and $t_{r+1}$ outside $T$ when they exist. A component of $D$ disjoint from $\{t_0\,,\dots,t_{r+1}\}$ is unaffected and remains a component of $D\triangle T$. Let $K_1\,,\dots,K_q$ be the other components of $D$, listed in order along $t_0t_1\cdots t_{r+1}\,$; for $\mathrm{Cr}_n\,$, one component may occur at both ends of this list, which only strengthens the following bound. If $q\ge2$, then the $q-1$ gaps between consecutive $K_i$ lie in $T$, so they are contained in $D\triangle T$; they lie in distinct components of $D\triangle T$, because $K_2\,,\dots,K_{q-1}$ lie in $T$ and so are disjoint from $D\triangle T$, and these components are distinct from the unaffected ones. Hence $\nu(D\triangle T)\ge \nu(D)-q+(q-1)$, and the same bound holds trivially if $q\le1$.

Finally, a subset of a path on $n$ vertices has at most $\lceil n/2\rceil$ components, and a subset of a cycle of length $2n$ has at most $n$; equality holds for the antichain $\min P$, namely $\{a_1\,,a_3\,,\dots\}$ for $Z_n$ and $\{a_1\,,\dots,a_n\}$ for $\mathrm{Cr}_n\,$, which is therefore at distance $w(P)$ from $\emptyset$. Because $P$ has no three-element chain, the graphs of $\Ocal(P)$ and $\Ccal(P)$ are isomorphic (Remark~\ref{rem:limitations}).
\end{proof}

An earlier version of this paper conjectured that $|\operatorname{diam}\Ocal(P)-\operatorname{diam}\Ccal(P)|\le1$ for every connected finite poset $P$; this conjecture has been disproved by \citet{Goel}.

%=====================================================================
\section{Looking forward}
%=====================================================================

\begin{conjecture}[{\cite[Conjecture~2.4]{HL-unimod}}]\label{conj:HL}
Let $P$ be a finite poset with $|P|=d\ge2$.
\begin{enumerate}[(b)]
\renewcommand{\theenumi}{\alph{enumi}}\renewcommand{\labelenumi}{\textup{(\theenumi)}}
\item $f_k(\Ocal(P))\le f_k(\Ccal(P))$ for $1\le k\le d-1$.
\item If $f_k(\Ocal(P))=f_k(\Ccal(P))$ for some $1\le k\le d-1$, then $\Ocal(P)$ and $\Ccal(P)$ are unimodularly equivalent.
\end{enumerate}
\end{conjecture}

Part~(a) was known for $k=1$ (Theorem~\ref{thm:HLSS}), $k=2$ (see \cite{FLM}) and $k=d-1$ (see \cite{HL-unimod}), and for all $k$ for several classes of posets (see \cite{AFJ,FL}); it has recently been proved for all $k$ in \cite{Fourier}. Part~(b) is false for $k=1$, because by Theorem~\ref{thm:HLSS} the two polytopes always have the same number of edges, whereas they are unimodularly equivalent only when $P$ is $X$-free (Theorem~\ref{thm:X}); accordingly, part~(b) is restated in \cite{FLM} for $2\le k\le d-1$. For $k\ge2$, part~(b) is known for $k=2$ (see \cite[Theorem~6.1]{FLM}) and $k=d-1$ (see \cite{HL-unimod}), and is open for $3\le k\le d-2$. Theorem~\ref{thm:main} and Freij-Hollanti, Lundstr\"om and Mori~\cite{FLM} suggest indexing the $k$-faces of both polytopes by a connected convex subposet $Q$ together with an antichain of $Q^{\parallel}$, and comparing local invariants of $Q$.

\backmatter

\section*{Declarations}

\begin{itemize}
\item \textbf{Funding.} J.~Lee was supported in part by the Office of Naval Research under grant \mbox{N00014-24-1-2694}, and by the Air Force Office of Scientific Research under award number \mbox{FA9550-24-1-0240}.
\item \textbf{Competing interests.} The author has no relevant financial or non-financial interests to disclose.
\item \textbf{Data availability.} Data sharing is not applicable to this article, as no datasets were generated or analysed during the current study.
\item \textbf{Use of generative AI.} During the preparation of this work, the author used Claude to assist in developing and checking the work. The author takes full responsibility for the content of the published article.
\end{itemize}

% Bibliography style sn-mathphys-num is set by the class option.
\bibliography{refs}

@article{AFJ,
  author  = {Ahmad, Ibrahim and Fourier, Ghislain and Joswig, Michael},
  title   = {Order and chain polytopes of maximal ranked posets},
  journal = {Discrete Applied Mathematics},
  volume  = {390},
  pages   = {28--39},
  year    = {2026},
  doi     = {10.1016/j.dam.2026.03.036}
}

@article{Chvatal,
  author  = {Chv{\'a}tal, V{\'a}clav},
  title   = {On certain polytopes associated with graphs},
  journal = {Journal of Combinatorial Theory, Series B},
  volume  = {18},
  pages   = {138--154},
  year    = {1975},
  doi     = {10.1016/0095-8956(75)90041-6}
}

@article{FL,
  author  = {Freij-Hollanti, Ragnar and Lundstr{\"o}m, Teemu},
  title   = {$f$-vector inequalities for order and chain polytopes},
  journal = {Mathematica Scandinavica},
  volume  = {130},
  number  = {3},
  pages   = {467--486},
  year    = {2024},
  doi     = {10.7146/math.scand.a-143995}
}

@article{FLM,
  author  = {Freij-Hollanti, Ragnar and Lundstr{\"o}m, Teemu and Mori, Aki},
  title   = {Two-dimensional faces of order and chain polytopes},
  journal = {Electronic Journal of Combinatorics},
  volume  = {33},
  number  = {3},
  eid     = {P3.67},
  year    = {2026},
  doi     = {10.37236/14909}
}

@inproceedings{Geissinger,
  author    = {Geissinger, Ladnor},
  title     = {The face structure of a poset polytope},
  booktitle = {Proceedings of the Third Caribbean Conference on Combinatorics and Computing},
  publisher = {University of the West Indies},
  address   = {Cave Hill, Barbados},
  pages     = {125--133},
  year      = {1981}
}

@book{GLS,
  author    = {Gr{\"o}tschel, Martin and Lov{\'a}sz, L{\'a}szl{\'o} and Schrijver, Alexander},
  title     = {Geometric Algorithms and Combinatorial Optimization},
  publisher = {Springer},
  address   = {Berlin},
  year      = {1988},
  doi     = {10.1007/978-3-642-97881-4}
}

@article{HL-unimod,
  author  = {Hibi, Takayuki and Li, Nan},
  title   = {Unimodular equivalence of order and chain polytopes},
  journal = {Mathematica Scandinavica},
  volume  = {118},
  pages   = {5--12},
  year    = {2016},
  doi     = {10.7146/math.scand.a-23291}
}

@article{HL-cutting,
  author  = {Hibi, Takayuki and Li, Nan},
  title   = {Cutting convex polytopes by hyperplanes},
  journal = {Mathematics},
  volume  = {7},
  number  = {5},
  pages   = {Article 381},
  year    = {2019},
  doi     = {10.3390/math7050381}
}

@article{HLSS,
  author  = {Hibi, Takayuki and Li, Nan and Sahara, Yoshimi and Shikama, Akihiro},
  title   = {The numbers of edges of the order polytope and the chain polytope of a finite partially ordered set},
  journal = {Discrete Mathematics},
  volume  = {340},
  pages   = {991--994},
  year    = {2017},
  doi     = {10.1016/j.disc.2017.01.005}
}

@article{LL,
  author  = {Lundstr{\"o}m, Teemu and {Saud Maia Leite}, Leonardo},
  title   = {Order polytopes of crown posets},
  journal = {European Journal of Combinatorics},
  volume  = {133},
  pages   = {Article 104304},
  year    = {2026},
  doi     = {10.1016/j.ejc.2025.104304}
}

@article{MO,
  author  = {Melamed, Michal and Onn, Shmuel},
  title   = {Convex integer optimization by constantly many linear counterparts},
  journal = {Linear Algebra and its Applications},
  volume  = {447},
  pages   = {88--109},
  year    = {2014},
  doi     = {10.1016/j.laa.2014.01.007}
}

@article{OR,
  author  = {Onn, Shmuel and Rothblum, Uriel G.},
  title   = {Convex combinatorial optimization},
  journal = {Discrete \& Computational Geometry},
  volume  = {32},
  pages   = {549--566},
  year    = {2004},
  doi     = {10.1007/s00454-004-1138-y}
}

@article{PZ,
  author  = {Petersen, T. Kyle and Zhuang, Yan},
  title   = {Zig-zag {E}ulerian polynomials},
  journal = {European Journal of Combinatorics},
  volume  = {124},
  pages   = {Article 104073},
  year    = {2025},
  doi     = {10.1016/j.ejc.2024.104073}
}

@article{Picard,
  author  = {Picard, Jean-Claude},
  title   = {Maximal closure of a graph and applications to combinatorial problems},
  journal = {Management Science},
  volume  = {22},
  pages   = {1268--1272},
  year    = {1976},
  doi     = {10.1287/mnsc.22.11.1268}
}

@article{Stanley,
  author  = {Stanley, Richard P.},
  title   = {Two poset polytopes},
  journal = {Discrete \& Computational Geometry},
  volume  = {1},
  pages   = {9--23},
  year    = {1986},
  doi     = {10.1007/BF02187680}
}

@article{VTL,
  author  = {Valdes, Jacobo and Tarjan, Robert E. and Lawler, Eugene L.},
  title   = {The recognition of series parallel digraphs},
  journal = {SIAM Journal on Computing},
  volume  = {11},
  pages   = {298--313},
  year    = {1982},
  doi     = {10.1137/0211023}
}

@misc{Fourier,
  author  = {Fourier, Ghislain},
  title   = {The {H}ibi--{L}i face-number conjecture},
  year    = {2026},
  howpublished = {arXiv:2609.23629},
  doi     = {10.48550/arXiv.2609.23629}
}

@book{EC1,
  author    = {Stanley, Richard P.},
  title     = {Enumerative Combinatorics, Volume 1},
  edition   = {Second},
  series    = {Cambridge Studies in Advanced Mathematics},
  volume    = {49},
  publisher = {Cambridge University Press},
  address   = {Cambridge},
  year      = {2012},
  doi     = {10.1017/CBO9781139058520}
}

@article{IT,
  author  = {Ikebe, Yoshiko T. and Tamura, Akihisa},
  title   = {Ideal polytopes and face structures of some combinatorial optimization problems},
  journal = {Mathematical Programming},
  volume  = {71},
  pages   = {1--15},
  year    = {1995},
  doi     = {10.1007/BF01592241}
}

@article{HK,
  author  = {Hausmann, Dirk and Korte, Bernhard},
  title   = {Colouring criteria for adjacency on 0--1-polyhedra},
  journal = {Mathematical Programming Study},
  volume  = {8},
  pages   = {106--127},
  year    = {1978},
  doi     = {10.1007/BFb0121197}
}

@book{Hausmann,
  author    = {Hausmann, Dirk},
  title     = {Adjacency on Polytopes in Combinatorial Optimization},
  series    = {Mathematical Systems in Economics},
  volume    = {49},
  publisher = {Anton Hain},
  address   = {Meisenheim am Glan},
  year      = {1980}
}

@misc{ABBLS,
  author       = {Aliniaeifard, Farid and Benedetti, Carolina and Bergeron, Nantel and Li, Shu Xiao and Saliola, Franco},
  title        = {Stable set polytopes and their 1-skeleta},
  howpublished = {arXiv:1804.00360},
  year         = {2021},
  doi          = {10.48550/arXiv.1804.00360}
}

@misc{Goel,
  author       = {Goel, Dave},
  title        = {Lee's conjecture on the diameters of order and chain polytopes fails in both directions},
  howpublished = {Preprint, Zenodo},
  year         = {2026},
  doi          = {10.5281/zenodo.23351399}
}

\end{document}